\documentclass[11pt,twoside,reqno,centertags,draft]{amsart}
\usepackage{amsfonts}
\usepackage{color,enumitem,graphicx}
\usepackage[colorlinks=true,urlcolor=blue,
citecolor=red,linkcolor=blue,linktocpage,pdfpagelabels,
bookmarksnumbered,bookmarksopen]{hyperref}

  \usepackage{amsmath,amsthm,amsfonts,amssymb}
\begin{document}

\title{Existence and mass concentration of pseudo-relativistic Hartree equation}
\date{}
\maketitle

\vspace{ -1\baselineskip}

{\small
\begin{center}
{\sc  Jianfu Yang} \\
Department of Mathematics,
Jiangxi Normal University\\
Nanchang, Jiangxi 330022,
P.~R.~China\\
email: Jianfu Yang: jfyang\_2000@yahoo.com\\[10pt]
{\sc  Jinge Yang} \\
School of Sciences,
Nanchang Institute of Technology\\
Nanchang 330099,
P.~R.~China\\
email: Jinge Yang: jgyang2007@yeah.net\\[10pt]

\end{center}
}

\renewcommand{\thefootnote}{}

\begin{quote}
{\bf Abstract.} In this paper, we investigate  the constrained minimization problem
\begin{equation}\label{eq:0.1}
e(a):=\inf_{\{u\in \mathcal{H},\|u\|_2^2=1\}}E_a(u),
\end{equation}
where the energy functional
\begin{equation} \label{eq:0.2}
E_a(u)=\int_{\mathbb{R}^3}(u\sqrt{-\Delta+m^2}\,u+Vu^2)\,dx
-\frac{a}{2}\int_{\mathbb{R}^3}(|x|^{-1}*u^2)u^2\,dx
\end{equation}
with $m\in \mathbb{R}$, $a>0$, is defined on a Sobolev space $\mathcal{H}$. We show that there exists a
threshold $a^*>0$ so that $e(a)$ is achieved if $0<a<a^*$, and has no minimizers if $a\geq a^*$. We also
investigate the asymptotic behavior of nonnegative minimizers of $e(a)$ as $a\to a^*$.
\end{quote}

\newcommand{\N}{\mathbb{N}}
\newcommand{\R}{\mathbb{R}}
\newcommand{\Z}{\mathbb{Z}}

\newcommand{\cA}{{\mathcal A}}
\newcommand{\cB}{{\mathcal B}}
\newcommand{\cC}{{\mathcal C}}
\newcommand{\cD}{{\mathcal D}}
\newcommand{\cE}{{\mathcal E}}
\newcommand{\cF}{{\mathcal F}}
\newcommand{\cG}{{\mathcal G}}
\newcommand{\cH}{{\mathcal H}}
\newcommand{\cI}{{\mathcal I}}
\newcommand{\cJ}{{\mathcal J}}
\newcommand{\cK}{{\mathcal K}}
\newcommand{\cL}{{\mathcal L}}
\newcommand{\cM}{{\mathcal M}}
\newcommand{\cN}{{\mathcal N}}
\newcommand{\cO}{{\mathcal O}}
\newcommand{\cP}{{\mathcal P}}
\newcommand{\cQ}{{\mathcal Q}}
\newcommand{\cR}{{\mathcal R}}
\newcommand{\cS}{{\mathcal S}}
\newcommand{\cT}{{\mathcal T}}
\newcommand{\cU}{{\mathcal U}}
\newcommand{\cV}{{\mathcal V}}
\newcommand{\cW}{{\mathcal W}}
\newcommand{\cX}{{\mathcal X}}
\newcommand{\cY}{{\mathcal Y}}
\newcommand{\cZ}{{\mathcal Z}}

\newcommand{\abs}[1]{\lvert#1\rvert}
\newcommand{\xabs}[1]{\left\lvert#1\right\rvert}
\newcommand{\norm}[1]{\lVert#1\rVert}

\newcommand{\loc}{\mathrm{loc}}
\newcommand{\p}{\partial}
\newcommand{\h}{\hskip 5mm}
\newcommand{\ti}{\widetilde}
\newcommand{\D}{\Delta}
\newcommand{\e}{\epsilon}
\newcommand{\bs}{\backslash}
\newcommand{\ep}{\emptyset}
\newcommand{\su}{\subset}
\newcommand{\ds}{\displaystyle}
\newcommand{\ld}{\lambda}
\newcommand{\vp}{\varphi}
\newcommand{\wpp}{W_0^{1,\ p}(\Omega)}
\newcommand{\ino}{\int_\Omega}
\newcommand{\bo}{\overline{\Omega}}
\newcommand{\ccc}{\cC_0^1(\bo)}
\newcommand{\iii}{\opint_{D_1}D_i}

\theoremstyle{plain}
\newtheorem{Thm}{Theorem}[section]
\newtheorem{Lem}[Thm]{Lemma}
\newtheorem{Def}[Thm]{Definition}
\newtheorem{Cor}[Thm]{Corollary}
\newtheorem{Prop}[Thm]{Proposition}
\newtheorem{Rem}[Thm]{Remark}
\newtheorem{Ex}[Thm]{Example}

\numberwithin{equation}{section}
\newcommand{\meas}{\rm meas}
\newcommand{\ess}{\rm ess} \newcommand{\esssup}{\rm ess\,sup}
\newcommand{\essinf}{\rm ess\,inf} \newcommand{\spann}{\rm span}
\newcommand{\clos}{\rm clos} \newcommand{\opint}{\rm int}
\newcommand{\conv}{\rm conv} \newcommand{\dist}{\rm dist}
\newcommand{\id}{\rm id} \newcommand{\gen}{\rm gen}
\newcommand{\opdiv}{\rm div}

\vskip 0.2cm \arraycolsep1.5pt
\newtheorem{Lemma}{Lemma}[section]
\newtheorem{Theorem}{Theorem}[section]
\newtheorem{Definition}{Definition}[section]
\newtheorem{Proposition}{Proposition}[section]
\newtheorem{Remark}{Remark}[section]
\newtheorem{Corollary}{Corollary}[section]

\section{Introduction}

In this paper, we investigate a minimization problem in connection with the pseudo-relativistic
Hartree equation
\begin{equation}\label{eq:1.1}
\sqrt{-\Delta+m^2}\,u+V(x)u= \mu u + a(|x|^{-1}*u^2)u\quad {\rm in}\quad \mathbb{R}^3,\quad  u\in H^{\frac{1}{2}}(\mathbb{R}^3),
\end{equation}
where $\mu\in \mathbb{R}$, and the operator $\sqrt{-\Delta+m^2}$ is defined on $H^{\frac 12}(\mathbb{R}^3)$ by the Fourier transform, that is,
for $u\in H^{\frac 12}(\mathbb{R}^3)$ we define
\[
\sqrt{-\Delta+m^2}\,u = \mathcal{F}^{-1}(\sqrt{|\xi|^2+m^2}\mathcal{F}u),
\]
in which $\mathcal{F}u$ stands for the Fourier transform of $u$. Problem \eqref{eq:1.1} arises in the study of solitary wave solutions $\psi(t,x)= e^{-i\mu t}u(x)$ of the Hartree equation
\begin{equation}\label{eq:1.2}
i\frac{\partial \psi}{\partial t} = \sqrt{-\Delta+m^2}\,\psi+V(x)\psi-a(|x|^{-1}*\psi^2)\psi\quad {\rm on}\quad \mathbb{R}^3.
\end{equation}
A consideration of problem \eqref{eq:1.2} in physics can be interpreted as a system of $N$ spinless, identical bosons with two-body interactions governed by the Coulomb potential. These bosons are also subject to a time-independent external potential $V(x)$, see \cite{FL} for more details. In the particular case $V(x)=-m$, problem \eqref{eq:1.2} was studied in \cite{ES} as an effective dynamical description for an N-body quantum system of relativistic bosons with two-body interaction given by Newtonian gravity, it leads to a Chandrasekhar type theory of boson stars. For solitary waves of problem \eqref{eq:1.2} with $V(x)=-m$, a ground state is a minimizer of the energy functional
\begin{equation}\label{eq:1.3a}
\mathcal{E}(u) = \frac 12\int_{\mathbb{R}^3}u(\sqrt{-\Delta+m^2}-m)u\,dx
-\frac{1}{4}\int_{\mathbb{R}^3}(|x|^{-1}*u^2)u^2\,dx
\end{equation}
constrained on
$$
\mathcal{N}(u) = \int_{\mathbb{R}^3}|u|^2\,dx = N,
$$
that is, a minimizer of the problem
\[
\mathcal{E}(N) = \inf\{\mathcal{E}(u): \mathcal{N}(u) = N\}.
\]
Any ground state $u$ satisfies
the semi-relativistic Hartree equation
\begin{equation}\label{eq:1.3}
(\sqrt{-\Delta+m^2}\,u - m) u - (|x|^{-1}*u^2)u = \mu u \quad {\rm in}\quad \mathbb{R}^3
\end{equation}
for some $\mu\in \mathbb{R}$. Equation \eqref{eq:1.3} appears in the study of models of stellar collapse, such as neutron stars. The typical neutron kinetic energy is high, so it must be treated relativistically, see \cite{L}, \cite{LT} and \cite{LY} for further discussion. It was found in \cite{FJL,LY} a symmetric decreasing ground state of $\mathcal{E}(u)$ subject to $0<N<N^*$ for some $N^*>0$; whereas no ground states exist whenever $N\geq N^*$.

Similar phenomenon also appears in quantum Bose gases. Recently, the mass concentration for Bose-Einstein condensates with attractive interaction was considered in \cite{GS} in $\mathbb{R}^2$. It proved in \cite{GS} that there exists a critical value $a^*>0$ such that the Gross-Pitaevskii energy functional
\[
E_a^{GP}(u) = \int_{\mathbb{R}^2}(|\nabla u|^2 + V(x)u^2)\,dx - \frac a 2\int_{\mathbb{R}^2}|u|^4\,dx
\]
with a "confining trap" $V$ and under the mass constraint $\int_{\mathbb{R}^2}u^2\,dx=1$, has at least one  minimizer if $0<a < a^*$, and has no minimizer if $a\geq a^*$. The limit behavior of the minimizer $u_{a_k}$ as $a_k\to a^*$ is also investigated. Similar problems were considered in \cite{DLS} for the fractional Laplacian $(-\Delta)^s$ and in \cite{HL} for Choquard equation.

In this paper, we consider  the minimization problem
\begin{equation}\label{eq:1.4}
e(a):=\inf_{\{u\in \mathcal{H},\|u\|_2^2=1\}}E_a(u),
\end{equation}
where the energy functional
\begin{equation} \label{eq:1.5}
E_a(u)=\int_{\mathbb{R}^3}(u\,\sqrt{-\Delta+m^2}\,u+Vu^2)\,dx
-\frac{a}{2}\int_{\mathbb{R}^3}(|x|^{-1}*u^2)u^2\,dx
\end{equation}
with $m\in \mathbb{R}$, $a>0$, is defined on the space
$$
\mathcal{H}:=\Big\{u\in H^{\frac{1}{2}}\big(\mathbb{R}^3\big)\Big|\int_{\mathbb{R}^3}Vu^2dx<\infty\Big\}
$$
equipped with the norm
$$\|u\|_{\mathcal{H}}=\Big\{\int_{\mathbb{R}^3}\Big[|(-\Delta)^{\frac{1}{4}}u|^2+(1+V)u^2\Big]dx\Big\}^{\frac{1}{2}}.$$

Stimulated by previous works, we will show that there exists a threshold $a^*>0$ such that the minimization problem $e(a)$ has a ground state if $0<a<a^*$ and has no ground state if $a\geq a^*$. Furthermore, we will study the collapse concentration of ground state $u_a$ as $a\to a^*$, which may help to understand better the structure formation of bosonic matter.

We assume in this paper that the function $V:\mathbb{R}^3\to \mathbb{R}$ is locally bounded and satisfies
\begin{equation}\label{eq:1.6}
V(x)\to \infty\quad {\rm as}\quad |x|\to \infty; \quad \inf_{x\in\mathbb{R}^3}V(x)=0.
\end{equation}
The threshold $a^*$ is related to $L^2(\mathbb{R}^3)$ norm $\|\cdot\|_2$ of a positive ground state $Q$ of the following problem
\begin{equation}\label{eq:1.7}
\sqrt{-\Delta}\,u+u-(|x|^{-1}*u^2)u=0\quad {\rm in}\quad \mathbb{R}^3,\quad u\in H^{\frac{1}{2}}(\mathbb{R}^3).
\end{equation}
Precisely,  $a^*:=\|Q\|_2^2$. It is known from \cite{FL1,LY} that, up to translations, problem \eqref{eq:1.7} admits a  positive ground state $Q$, which can be taken to be radially symmetric about the origin.  Moreover, every ground state $Q$ of \eqref{eq:1.7} satisfies $Q\in H^s(\mathbb{R}^3)$ for all $s\ge\frac12$ and
\begin{equation}\label{eq:1.10}
Q(x)=O(|x|^{-4})
\end{equation}
as $|x|\rightarrow \infty$.

In describing the formation of nonnegative minimizers of $e(a)$ as $a\to a^*$,  the uniqueness of the ground state $Q$ of \eqref{eq:1.7} has vital function. However, such a problem is still open for \eqref{eq:1.7} , this causes difficulty to manage the problem. We observe in Lemma \ref{lem:A3} in section 4 that every ground state of \eqref{eq:1.7} has the least $L^2(\mathbb{R}^3)$ norm among all nontrivial solutions of \eqref{eq:1.7}. Fortunately, this is sufficient to serve our purpose.

On the other hand, by the  Gagliardo-Nirenberg inequality
\begin{equation}\label{eq:1.7a}
\|u\|_q\le \|u\|_2^\theta\|(-\Delta)^{\frac{1}{4}}u\|_2^{1-\theta},
\end{equation}
for $u\in H^{\frac{1}{2}}(\mathbb{R}^3)$, where $2\le q \le 3,\, \theta=\frac{6}{q}-2$, and
the Hardy-Littlehood-Soblev inequality, we have
\begin{equation}\label{eq:1.8}
\int_{\mathbb{R}^3}(|x|^{-1}*u^2)u^2\,dx\leq C_{gn}\int_{\mathbb{R}^3}|(-\Delta)^{\frac{1}{4}}u|^2\,dx\int_{\mathbb{R}^3}u^2\,dx
\end{equation}
for $u\in H^{\frac{1}{2}}(\mathbb{R}^3)$. The optimal constant in inequality \eqref{eq:1.8} was determined in \cite{FL1}. Taking into account as \cite{W} the Weinstein functional
\begin{equation}\label{eq:1.8a}
I(u)=\frac{\int_{\mathbb{R}^3} u\sqrt{-\Delta}u\,dx\int_{\mathbb{R}^3}u^2\,dx}
{\int_{\mathbb{R}^3}(|x|^{-1}*u^2)u^2\,dx},
\end{equation}
we know from Lemma 5 in \cite{LenLew}  that
\begin{equation}\label{eq:1.9}
\frac{1}{2}\int_{\mathbb{R}^3}Q^2\,dx=\inf_{u\in H^{\frac{1}{2}}(\mathbb{R}^3),u\not\equiv0 } I(u),
\end{equation}
and then, the best constant $C_{gn}$ in \eqref{eq:1.8} is equal to $\frac 2{\|Q\|^2_2}$.

Our first result is as follows.

\begin{Theorem} \label{thm:1.1} Let $V\in L_{loc}(\mathbb{R}^3)$ satisfying $\eqref{eq:1.6}$. Then,

(i) if $0\le a<a^*:=\|Q\|_2^2$, there exists at least one nonnegative minimizer $u_a$ for $e(a)$;

(ii) if $a\ge a^*$, there is no minimizer for $e(a)$.
\end{Theorem}

If $u_a$ is a minimizer of $e(a)$, it is natural to think that $|u_a|$ is also a minimizer of $e(a)$.
But this is not obvious since we are
dealing with nonlocal operators. We will show in section 4 that $E_a(u)\geq E_a(|u|)$ for $u\in \mathcal{H}$. So the minimizer $u_a$ does not change the sign, and satisfies the following equation
\begin{equation}\label{eq:1.10a}
\sqrt{-\Delta+m^2}\,u+V(x)u=\mu_a u + a(|x|^{-1}*u^2)u\quad {\rm in}\quad \mathbb{R}^3,
\end{equation}
where $\mu_a$ is the Lagrange multiplier.

Next, we investigate the asymptotic behavior of the nonnegative minimizer $u_a$ of $e(a)$ as $a\to a^*$. To localize concentration points of $u_a$ as $a\to a^*$, the graph of the external potential $V$ plays an important role. We assume that $V$ has isolated minima and behaves like a power of the distance from these points. Precisely,
we suppose that there exist $n\geq 1$ distinct points $x_i\in \mathbb{R}^3$ such that
\begin{equation}\label{eq:1.11}
V(x)=\prod_{i=1}^{n}|x-x_i|^{p_i}
\end{equation}
with $p_i>0$. Let $p=\max\{p_i|1\le i\le n\}$ and
set $\kappa_i\in (0,\infty]$ by
\begin{equation}\label{eq:1.12}
\kappa_i=\lim_{x\rightarrow x_i}\frac{V(x)}{|x-x_i|^p}.
\end{equation}
Define $\kappa=\min\{\kappa_i|1\le i\le n\}$ and set
\begin{equation}\label{eq:1.13}
\mathcal{Z}=\{x_i|\kappa_i=\kappa\}.
\end{equation}

\bigskip

Our second result is as follows.
\begin{Theorem}\label{thm:1.2}
Suppose $V$ satisfies \eqref{eq:1.11}. Let $u_a>0$ be a nonnegative minimizer of $e(a)$ with $a<a^*$. If one of the following conditions holds,

(i) $m\neq0$ and $0<p<1$;

(ii) $m=0$ and either $0<p<\frac 52$ or $0<\sum_{i=1}^np_i<5$,\\
then for any sequence $\{a_k\}$ with $a_k\rightarrow a^*$ as $k\rightarrow \infty$, there exists a subsequence of $\{a_k\}$, still denoted by $\{a_k\}$,
an $x_0\in \mathcal{Z}$, and a ground state $Q$  of \eqref{eq:1.7} such that
\begin{equation}\label{eq:1.14a}
\lim_{k\rightarrow \infty}(a^*-a_k)^{\frac{3}{2(p+1)}}u_{a_k}((a^*-a_k)^{\frac{1}{p+1}}x+x_0)=
\frac{\mu^{\frac{3}{2}}}{\|Q\|_2}Q(\mu x)
\end{equation}
strongly in $L^2(\mathbb{R}^3)$ with
\begin{equation}\label{eq:1.14}
\mu=\Big(p\kappa\int_{\mathbb{R}^3}|x|^pQ(x)^2\,dx\Big)^{\frac{1}{p+1}}.
\end{equation}
\end{Theorem}

Theorem \ref{thm:1.2} describes the collapse concentration of ground state $u_a$ as $a\to a^*$.
Since the operator $\sqrt{-\Delta+m^2}$ and nonlinear term are nonlocal, and the operator $\sqrt{-\Delta+m^2}$ is not scaling invariant, it brings new difficulties. To pursue the study of the problem, it needs to develop some new techniques. In treating the nonlocal term, we need to use the Hardy-Littlewood-Sobolev inequality. Since we can not use integral by part for fractional operators directly, instead we need to compare the operator $\sqrt{-\Delta+m^2}$ with the fractional Laplacian, and estimate the commutator of the fractional Laplacian and smooth functions.

Finally, we remark that although it is of interest to consider the same problem for the general operator $(-\Delta+m^2)^s$ in mathematics, we focus on the case $s = \frac 12$ for the sake of physical reasons, and our argument can be carried out for the general case.

After this paper was submitted, we learned that a similar minimization problem of \eqref{eq:1.4} was considered in \cite{GZ} and \cite{N}. In \cite{GZ}, a problem with nonlocal nonlinear terms was studied while in \cite{N}, it was considered a case without trapping potential $V$. This paper is  organized as follows. In section 2, we show the existence and nonexistence of minimizers of $e(a)$ and prove Theorem \ref{thm:1.1}; in section 3, we first estimate the energy of minimizers, then using the blow-up method to establish \eqref{eq:1.14a}. Hence, Theorem \ref{thm:1.2} follows.

In the sequel, we use $\|\cdot\|_p$ to denote $L^p(\mathbb{R}^3)$ norm.

\bigskip

\section{Existence and nonexistence of minimizers}

\bigskip

In this section, we will study the existence and nonexistence of  minimizers for the minimization problem $e(a)$.
To this end, we need to establish a compact Sobolev type embedding, which will be shown by the Galiardo-Nirenberg inequality.

\begin{Lemma} \label{lem:2.1} Suppose $V\in L_{loc}(\mathbb{R}^3)$ and  $\lim_{x\to\infty}V(x)=\infty$. Then the embedding $\mathcal{H}\hookrightarrow L^q(\mathbb{R}^3)$ is compact for $ q\in[2,3)$.
\end{Lemma}
\begin{proof} Let $\{u_n\}$ be a bounded sequence in $\mathcal{H}$. We may assume that
\[
u_n\rightharpoonup u \quad {\rm in}\quad H^{\frac{1}{2}}(\mathbb{R}^3)\quad {\rm and}\quad  u_n\to u\quad {\rm in}\quad L_{loc}^q(\mathbb{R}^3)
\]
as $n\to\infty$ for $2\le q<3$. For $R>0$, we deduce
\[
\begin{split}
 &\int_{\mathbb{R}^3}|u_n-u|^2\,dx\\
 \le &\int_{B_R}|u_n-u|^2\,dx+2\int_{\mathbb{R}^3\backslash B_R}u^2\,dx+2\int_{\mathbb{R}^3\backslash B_R}u_n^2\,dx\\
 \le &\int_{B_R}|u_n-u|^2\,dx+2\int_{\mathbb{R}^3\backslash B_R}u^2\,dx+2\sup_{x\in\mathbb{R}^3\backslash B_R}V(x)^{-1}\int_{\mathbb{R}^3\backslash B_R}Vu_n^2\,dx.
 \end{split}
\]
 Let $n\rightarrow \infty$, and then let $R\rightarrow\infty$, we obtain $u_n\rightarrow u$ in $L^2(\mathbb{R}^3)$.
 Since $\{u_n\}$ is bounded in $H^{\frac{1}{2}}(\mathbb{R}^3)$, it follows from \eqref{eq:1.7a} that $u_n\rightarrow u$ in $L^q(\mathbb{R}^3)$ for $2\le q<3$.
 \end{proof}

Next, we show that the $L^p$ convergence implies the convergence of the nonlocal term.

\begin{Lemma}\label{lem:2.2} If $u_n\to u$ in $L^q(\mathbb{R}^3)$ for $2\le q<3$, then
 \[
 \lim_{n\rightarrow\infty}\int_{\mathbb{R}^3}(|x|^{-1}*u_n^2)u_n^2\,dx=
 \int_{\mathbb{R}^3}(|x|^{-1}*u^2)u^2\,dx.
\]
 \end{Lemma}

 \begin{proof} Since $\{u_n\}$ is bounded in $L^q(\mathbb{R}^3)$ for $2\le q<3$, by the Hardy-Littlewood-Sobolev
 inequality and the H\"{o}lder inequality, we have
 \begin{align}\label{eq:2.1}
 &\Big|\int_{\mathbb{R}^3}(|x|^{-1}*u_n^2)u_n^2\,dx
 -\int_{\mathbb{R}^3}(|x|^{-1}*u^2)u^2\,dx\Big|\nonumber\\
 &\le\Big|\int_{\mathbb{R}^3}(|x|^{-1}*u_n^2)(u_n^2-u^2)\,dx\Big|+
 \Big|\int_{\mathbb{R}^3}(|x|^{-1}*(u_n^2-u^2)u^2\,dx\Big|\nonumber\\
 &\le C\|u_n^2\|_{\frac{6}{5}}\|u_n^2-u^2\|_{\frac{6}{5}}+C\|u_n^2-u^2\|_{\frac{6}{5}}
 \|u^2\|_{\frac{6}{5}}\nonumber\\
 &\le C\|u_n\|_{\frac{12}{5}}^2\|u_n+u\|_{\frac{12}{5}}\|u_n-u\|_{\frac{12}{5}}+
 C\|u_n+u\|_{\frac{12}{5}}\|u_n-u\|_{\frac{12}{5}}\|u\|_{\frac{12}{5}}^2.
 \nonumber\\
  &\le C\|u_n-u\|_{\frac{12}{5}}.
 \end{align}
 The assertion follows by letting $n\to \infty$ in \eqref{eq:2.1}.
 \end{proof}

 \bigskip

 {\bf Proof of Theorem \ref{thm:1.1}} We first show $(i)$ of Theorem \ref{thm:1.1}.
 It is well known that
 \begin{equation}\label{eq:2.2}
 \sqrt{-\Delta+m^2}\ge\sqrt{-\Delta},
 \end{equation}
 that is, for $u\in H^{\frac12}(\mathbb{R}^3)$,
 $$
 (\sqrt{-\Delta+m^2}\,u,u)\ge(\sqrt{-\Delta}u,u).
 $$

 In particular, if $u\in\mathcal{H}$  with $\|u\|_2^2=1$, we infer from the Gagliardo-Nirenberg inequality \eqref{eq:1.8} that
 \begin{align}\label{eq:2.3}
 E_a(u)&\ge\int_{\mathbb{R}^3}\big(|(-\Delta)^{\frac{1}{4}}u|^2+Vu^2\big)\,dx
 -\frac{a}{2}\int_{\mathbb{R}^3}\big(|x|^{-1}*u^2\big)u^2\,dx\nonumber\\
 &\ge \Big(1-\frac{a}{a^*}\Big)\int_{\mathbb{R}^3}|(-\Delta)^{\frac{1}{4}}u|^2\,dx
 +\int_{\mathbb{R}^3}Vu^2\,dx,
 \end{align}
that is, $E_a(u)$ is bounded from below if $0\le a\le a^*$.

Now, let $\{u_n\}\in \mathcal{H}$ be a minimizing sequence of $e(a)$, i.e., $\|u_n\|_2^2=1$, and
$$
\lim_{n\to\infty}E_a(u_n)=e(a).
$$

If $0\le a<a^*$, we see from \eqref{eq:2.3} that $\{u_n\}$ is uniformly bounded in $\mathcal{H}$. By Lemma \ref{lem:2.1}, we may assume that
$$
u_n\rightharpoonup u\quad {\rm in}\quad \mathcal{H}, \quad {\rm and}\quad  u_n\to u\quad {\rm in}\quad L^q(\mathbb{R}^3)
$$
for $2\le q<3$. Hence,  $$\int_{\mathbb{R}^3}u^2\, dx=1,$$
and then
\[
E_a(u)\geq e(a).
\]

On the other hand, by Lemma A.4 in \cite{FJL}, the functional $(\sqrt{-\Delta+m^2}\,u,u)$ is weakly lower semi-continuous, Lemma \ref{lem:2.2} yields that
$$
E_a(u)\le e(a).
$$
As a result, $u$ is a minimizer of $e(a)$, and Lemma \ref{lem:A2} implies that $|u|$ is also a minimizer of $e(a)$.

Next, we prove $(ii)$ of Theorem \ref{thm:1.1}. To this purpose, we will construct a function $U_R$ such that if $a>a^*$, $E_a(U_R)\to -\infty$ as $R\to\infty$.

Choose a non-negative function $\varphi\in C_0^\infty(\mathbb{R}^3)$ satisfying
\begin{align}\label{eq:2.4a}
\varphi\equiv1 ~~{\rm for}~~|x|\le \delta; \quad\varphi\equiv0 ~~{\rm for} ~~|x|\ge 2\delta; \quad0\le\varphi\le1,
\end{align}
where $\delta>0$. For $R>1$, and $x_0\in\mathbb{R}^3$, set
\begin{equation}\label{eq:2.4}
U_R=A_R\frac{R^{\frac{3}{2}}}{\|Q\|_2}\varphi(x-x_0)Q\big(R(x-x_0)\big)
\end{equation} with $\|U_R\|_2^2=1$,
that is,
\begin{equation}\label{eq:2.5}
1=\int_{\mathbb{R}^N}\frac{A_R^2R^3}{\|Q\|_2^2}\varphi^2(x)Q^2(Rx)\,dx=
\int_{\mathbb{R}^N}\frac{A_R^2}{\|Q\|_2^2}\varphi^2(R^{-1}x)Q^2(x)\,dx.
\end{equation}
Obviously,
\[
1\ge \int_{B_{R\delta}}\frac{A_R^2}{\|Q\|_2^2}Q^2(x)\,dx
\ge\int_{B_{\delta}}\frac{A_R^2}{\|Q\|_2^2}Q^2(x)\,dx.
\]
It yields
\[
A_R^2\le \bigg(\int_{B_{\delta}}\frac{Q^2(x)}{\|Q\|_2^2}\,dx\bigg)^{-1}.
\]
The decay law of $Q$ given in \eqref{eq:1.10} implies that for $R$ large enough,
 \begin{equation}\label{eq:2.6}
 \begin{split}
 &|1- A_R^2|\\
 =&\left|\int_{\mathbb{R}^3}\frac{A_R^2}{\|Q\|_2^2}\Big(
   \varphi^2\big(\frac xR\big)-1\Big)Q^2(x)\,dx\right|\\
   \le &\int_{\mathbb{R}^3\setminus B_{R\delta}}\left|\frac{A_R^2}{\|Q\|_2^2}\Big(
   \varphi^2\big(\frac xR\big)-1 \Big)Q^2(x)\right|\,dx\\
   \le& C\int_{\mathbb{R}^3\setminus B_{R\delta}}|x|^{-8}\,dx\\
  \le &C(\delta)R^{-5}.\\
\end{split}
  \end{equation}
Therefore,
  \begin{equation}\label{eq:2.7}
  1\leq A_R^2=1-\int_{\mathbb{R}^N}\frac{A_R^2}{\|Q\|_2^2}\Big(
   \varphi^2\big(\frac xR\big)-1\Big)Q^2(x)\,dx\le 1+O(R^{-5})
   \end{equation}
as $R\to\infty$.

Now, we treat translations and scaling of integrals involving the nonlocal operator $\sqrt{-\Delta+m^2}$.
We will show that
\begin{equation}\label{eq:2.8}
\begin{split}
&\frac{\|Q\|_2^2}{A_R^2}\int_{\mathbb{R}^3}U_R\sqrt{-\Delta+m^2}U_R\,dx\\
=&R\int_{\mathbb{R}^3}\varphi(R^{-1}x)Q(x)\sqrt{-\Delta+R^{-2}m^2}
\big(\varphi(R^{-1}x)Q(x)\big)\,dx.\\
\end{split}
\end{equation}

We deal with the translation first.
By the definition of the pseudo-differential operators $\sqrt{-\Delta+m^2}$, we have
\begin{align}\label{eq:2.9}
&\quad\frac{\|Q\|_2^2}{A_R^2}
\int_{\mathbb{R}^3}U_R\sqrt{-\Delta+m^2}U_R\,dx\nonumber\\
&=\int_{\mathbb{R}^3}R^3\varphi(x-x_0)Q(R(x-x_0))\mathcal{F}^{-1}\Big(\sqrt{\xi^2+m^2}\mathcal{F}
\big(\varphi(y-x_0)Q(R(y-x_0))\big)\Big)\,dx
\nonumber\\
&=\frac 1{(2\pi)^{3}}\int_{\mathbb{R}^3}R^3\varphi(x)Q(Rx)\,dx
\int_{\mathbb{R}^3}{\rm e}^{i(x+x_0)\cdot \xi}\sqrt{\xi^2+m^2}\,d\xi
\int_{\mathbb{R}^3}{\rm e}^{-i\xi\cdot (y+x_0)}\varphi(y)Q(Ry)\,dy\nonumber\\
&=\int_{\mathbb{R}^3}R^3\varphi(x)Q(Rx)\mathcal{F}^{-1}\big(\sqrt{\xi^2+m^2}
\mathcal{F}(\varphi(y)Q(Ry))\big)\,dx\nonumber\\
&=\int_{\mathbb{R}^3}R^3\varphi(x)Q(Rx)\sqrt{-\Delta+m^2}\big(\varphi(x)Q(Rx)\big)\,dx.
\end{align}
Next, for the scaling, by the Plancherel theorem, we deduce
\begin{align}\label{eq:2.10}
&\int_{\mathbb{R}^3}R^3\varphi(x)Q(Rx)\sqrt{-\Delta+m^2}\big(\varphi(x)Q(Rx)\big)\,dx
\nonumber\\
=&\int_{\mathbb{R}^3}R^3\sqrt{\xi^2+m^2}\mathcal{F}\big(\varphi(y)Q(Ry))\big)^2\,d\xi\nonumber\\
=&\frac 1{(2\pi)^{3}}R^3\int_{\mathbb{R}^3}\sqrt{\xi^2+m^2}d\xi
\Big(\int_{\mathbb{R}^3}{\rm e}^{-i\xi\cdot y}\varphi(y)Q(Ry)\,dy\Big)^2\nonumber\\
=&\frac 1{(2\pi)^{3}}R\int_{\mathbb{R}^3}\sqrt{\xi^2+R^{-2}m^2}d\xi
\Big(\int_{\mathbb{R}^3}{\rm e}^{-i\xi\cdot y}\varphi(R^{-1}y)Q(y)\,dy\Big)^2\nonumber\\
=&R\int_{\mathbb{R}^3}\sqrt{\xi^2+R^{-2}m^2}
\mathcal{F}\big(\varphi(R^{-1}y)Q(y)\big)^2\,d\xi\nonumber\\
=&R\int_{\mathbb{R}^3}\varphi(R^{-1}x)Q(x)\sqrt{-\Delta+R^{-2}m^2}
\big(\varphi(R^{-1}x)Q(x)\big)\,dx.
\end{align}
Equations \eqref{eq:2.9} and \eqref{eq:2.10} implies \eqref{eq:2.8}.

By Lemma 3 in \cite{LY} or Lemma \ref{lem:A1} in Appendix, there holds
\begin{equation}\label{eq:2.11}
\sqrt{-\Delta+R^{-2}m^2}\le \sqrt{-\Delta}+\frac{1}{2}R^{-2}m^2(-\Delta)^{-\frac{1}{2}}.
\end{equation}
Equations \eqref{eq:2.8} and \eqref{eq:2.11} yield that
\begin{align}\label{eq:2.12}
&\quad\frac{\|Q\|_2^2}{A_R^2}\int_{\mathbb{R}^3} U_R\sqrt{-\Delta+m^2}U_R\,dx\nonumber\\
&\le R\int_{\mathbb{R}^3}\varphi(R^{-1}x)Q(x)\sqrt{-\Delta}\Big(\varphi(R^{-1}x)
Q(x)\Big)\,dx\nonumber\\
&~~~~+\frac{m^2}{2}R^{-1}\int_{\mathbb{R}^3}\varphi(R^{-1}x)Q(x)
(-\Delta)^{-\frac{1}{2}}\Big(\varphi(R^{-1}x)
Q(x)\Big)\,dx\nonumber\\
&=R\biggl(\int_{\mathbb{R}^3}Q(x)\sqrt{-\Delta}Q(x)\,dx+\int_{\mathbb{R}^3}
\Big(\varphi(R^{-1}x)-1\Big)Q(x)\sqrt{-\Delta}Q(x)\,dx\nonumber\\
&~~+\int_{\mathbb{R}^3}
\varphi(R^{-1}x)Q(x)\sqrt{-\Delta}\left(\Big(\varphi(R^{-1}x)-1\Big)Q(x)
\right)\,dx\biggl)\nonumber\\
&~~+\frac{m^2}{2}R^{-1}\int_{\mathbb{R}^3}\varphi\big(R^{-1}x\big)Q(x)
(-\Delta)^{-\frac{1}{2}}\Big(\varphi\big(R^{-1}x\big)Q(x)\Big)\,dx\nonumber\\
&=:R\int_{\mathbb{R}^3}Q(x)\sqrt{-\Delta}Q(x)\,dx+RA+RB+\frac{m^2}{2}R^{-1}E.
\end{align}

In the following, we estimate terms $A, B$ and $E$. We commence with the term $A$. By the Newton theorem, see Theorem 9.7 of \cite{LL},
\begin{equation}\label{eq:2.13}
|x|^{-1}*Q^2\le \frac{\|Q\|_2^2}{|x|},
\end{equation}
since $Q$ is a solution of problem \eqref{eq:1.7}, we find that
 \begin{equation}\label{eq:2.14}
 |\sqrt{-\Delta}Q|\le\Big(\frac{\|Q\|_2^2}{|x|}+1\Big)|Q|.
\end{equation}
Hence, we estimate by \eqref{eq:1.10} that
\begin{align}\label{eq:2.15}
|A|&\le\int_{\mathbb{R}^3\backslash B_{R\delta}}\Big|\big(\varphi(R^{-1}x)-1\big) Q(x)\sqrt{-\Delta}Q(x)\Big|\,dx
 \nonumber\\
 &\le C\int_{\mathbb{R}^3\backslash B_{R\delta}}|x|^{-8}dx
 \nonumber\\
 &\le C(\delta)R^{-5}.
\end{align}

To estimate $B$, we recall an estimate for commutators. Denote $\varphi_R:=\varphi(R^{-1}x)$. It is known from \cite{Cal,CoMe}, see also Remark 4 in \cite{LenLew}, that there holds
\begin{equation}\label{eq:2.16}
\|[\sqrt{-\Delta},\varphi_R]\|_{L^2(\mathbb{R}^3)}\le C\|\nabla \varphi_R\|_\infty,
\end{equation}
where $[\sqrt{-\Delta},\varphi_R]$ is the commutator of $\sqrt{-\Delta}$ and $\varphi_R$, and $[\cdot,\cdot]$ is the Lie bracket. Therefore,
\begin{equation}\label{eq:2.17}
\int_{\mathbb{R}^3}\big([\sqrt{-\Delta},\varphi_R]Q\big)^2\,dx\le C\|\nabla \varphi_R\|_\infty^2\|Q\|_2^2.
\end{equation}
Apparently,
\[
\begin{split}
|B|&\le\Big|\int_{\mathbb{R}^3}\big(\varphi(R^{-1}x)-1\big)Q(x)\varphi(R^{-1}x)
\sqrt{-\Delta} Q(x)\,dx\Big|\\
&\quad+\Big|\int_{\mathbb{R}^3}\big(\varphi(R^{-1}x)-1\big)Q(x)
[\sqrt{-\Delta},\varphi_R]Q(x)\,dx\Big|\\
&\le C\int_{\mathbb{R}^3\backslash B_{R\delta}}\big(|Q(x)|
|\sqrt{-\Delta} Q(x)|+|Q(x)|
\big|[\sqrt{-\Delta},\varphi_R]Q(x)\big|\big)\,dx\\
& = B_1 + B_2.
\end{split}
\]
By \eqref{eq:1.10} and \eqref{eq:2.14},
\[
|B_1|\leq C\int_{\mathbb{R}^3\backslash B_{R\delta}}|x|^{-8}\,dx\leq C(\delta)R^{-5}.
\]
The H\"older inequality and \eqref{eq:2.17} give that
\[
\begin{split}
|B_2|
&\leq C\Big(\int_{\mathbb{R}^3\backslash B_{R\delta}}Q(x)^2\,dx\Big)^{\frac{1}{2}}
\|[\sqrt{-\Delta},\varphi_R]Q\|_2\\
&\le C\Big(\int_{\mathbb{R}^3\backslash B_{R\delta}}|x|^{-8}dx\Big)^{\frac{1}{2}}\|\nabla\varphi_R\|_\infty\|Q\|_2\\
&\le C(\delta)R^{-\frac{5}{2}}R^{-1}\|Q\|_2^2\\
&\le C(\delta)R^{-\frac{7}{2}}.
\end{split}
\]
Finally, it follows from the Hardy-Littlewood-Sobolev inequality that
\[
|E|\le C\Big\|\varphi_{R}Q\Big\|^2_{\frac{3}{2}}\le C\|Q\|^2_{\frac{3}{2}}\le C\|Q\|^2_{H^{\frac{1}{2}}(\mathbb{R}^3)}.
\]
As a result,
\begin{align}\label{eq:2.18}
&\int_{\mathbb{R}^3}U_R\sqrt{-\Delta+m^2}U_R\,dx\nonumber\\
&\le\frac{R}{\|Q\|_2^2}\int_{\mathbb{R}^3}Q\sqrt{-\Delta}Q\,dx
+O(R^{-\frac{5}{2}})+ m^2O(R^{-1}).\nonumber\\
\end{align}

Now, we turn to consider the nonlocal term,
\[
\begin{split}
&\int_{\mathbb{R}^3}(|x|^{-1}*U_R^2)U_R^2\,dx\\
&=\frac{A_R^4R^6}{\|Q\|_2^4}\int_{\mathbb{R}^3}\int_{\mathbb{R}^3}
\frac{1}{|x-y|}\varphi^2(y)Q^2(Ry)\varphi^2(x)Q^2(Rx)\,dxdy\\
&=\frac{A_R^4R}{\|Q\|_2^4}\int_{\mathbb{R}^3}\int_{\mathbb{R}^3}
\frac{1}{|x-y|}\varphi^2(R^{-1}y)Q^2(y)\varphi^2(R^{-1}x)Q^2(x)\,dxdy\\
&=:\frac{A_R^4R}{\|Q\|_2^4}F.
\end{split}
\]
We can write
\[
\begin{split}
F&=\int_{\mathbb{R}^3}\int_{\mathbb{R}^3}\frac{1}{|x-y|}Q^2(x)Q^2(y)\,dxdy\\
&\quad+\int_{\mathbb{R}^3}\int_{\mathbb{R}^3}\frac{1}{|x-y|}\big(\varphi^2(R^{-1}x)-1\big)Q^2(x)Q^2(y)\,dxdy\\
&\quad+\int_{\mathbb{R}^3}\int_{\mathbb{R}^3}\frac{1}{|x-y|}\big(\varphi^2(R^{-1}y)-1\big)Q^2(y)\varphi^2(R^{-1}x)Q^2(x)\,dydx\\
&=:\int_{\mathbb{R}^3}(|x|^{-1}*Q^2)Q^2dx+F_1+F_2.
\end{split}
\]
By \eqref{eq:1.10} and the Newton theorem,
\[
\begin{split}
|F_1|
&\le \int_{\mathbb{R}^3}\int_{\mathbb{R}^3}\frac{1}{|y|}
\big|\varphi^2(R^{-1}x)-1\big|Q^2(x)Q^2(y)\,dxdy\\
&\le C\int_{\mathbb{R}^3}\frac{1}{|y|}\int_{\mathbb{R}^3\setminus B_{R\delta}}Q^2(x)Q^2(y)\,dxdy\\
&\le C\int_{\mathbb{R}^3}\int_{\mathbb{R}^3\setminus B_{R\delta}}\frac{1}{|y|}|x|^{-8}Q^2(y)\,dxdy\\
&\le C(\delta)R^{-5}\int_{\mathbb{R}^3}\frac{Q(y)^2}{|y|}\,dy.\\
\end{split}
\]
The Hardy inequality yields
\[
|F_1|\le C(\delta)R^{-5}\|(-\Delta)^{\frac{1}{4}}Q\|_2\le C(\delta)R^{-5}.
\]
Similarly, we have
\[
|F_2|\le  C(\delta)R^{-5}.
\]
In conclusion,
\begin{equation}\label{eq:2.19}
\begin{split}
&\int_{\mathbb{R}^3}(|x|^{-1}*U_R^2)U_R^2\,dx\\
\geq&\frac{R}{\|Q\|_2^4}\Big(\int_{\mathbb{R}^3}(|x|^{-1}*Q^2)Q^2\,dx-|F_1|-|F_2|\Big)\\
\geq &\frac{R}{\|Q\|_2^4}\int_{\mathbb{R}^3}(|x|^{-1}*Q^2)Q^2\,dx-O(R^{-4})\\
\end{split}
\end{equation}
for $R>0$ large.
By Lemma 5 in \cite{LenLew},
\[
2\|(-\Delta)^{\frac{1}{4}}Q\|_2^2=
\int_{\mathbb{R}^3}(|x|^{-1}*Q^2)
Q^2\,dx,
\]
we obtain from \eqref{eq:2.18} and \eqref{eq:2.19} that
\begin{align}\label{eq:2.20}
&\int_{\mathbb{R}^3}U_R\sqrt{-\Delta+m^2}U_R\,dx-\frac{a}{2}
\int_{\mathbb{R}^3}(|x|^{-1}*U_R^2)U_R^2d\,x\nonumber\\
&\le \frac{R}{\|Q\|_2^2}\Big(\|(-\Delta)^{\frac{1}{4}}Q\|_2^2
-\frac{a}{2\|Q\|_2^2}\int_{\mathbb{R}^3}(|x|^{-1}*Q^2)
Q^2\,dx\Big)\nonumber\\
&+O(R^{-\frac{5}{2}})+m^2O(R^{-1})\nonumber\\
&=\frac{R}{2\|Q\|_2^2}\Big(1
-\frac{a}{\|Q\|_2^2}\Big)\int_{\mathbb{R}^3}(|x|^{-1}*Q^2)
Q^2\,dx+O(R^{-\frac{5}{2}})+m^2O(R^{-1}).
\end{align}
On the other hand, the dominated convergence theorem implies
\begin{align}\label{eq:2.21}\lim_{R\rightarrow\infty}\int_{\mathbb{R}^3}VU_R^2\,dx
&=\lim_{R\to\infty}\int_{\mathbb{R}^3}V(x)\frac{A_R^2R^3}{\|Q\|_2^2}\varphi^2(x-x_0)Q^2(R(x-x_0))\,dx
\nonumber\\
&=\lim_{R\rightarrow\infty}\frac{A_R^2}{\|Q\|_2^2}\int_{\mathbb{R}^3}V(R^{-1}x+x_0)\varphi^2(R^{-1}x)Q^2(x)\,dx
\nonumber\\
&=V(x_0).
\end{align}

Now, we consider separately the cases $a>a^*=\|Q\|_2^2$ and $a=a^*$.

If  $a>a^*$, by \eqref{eq:2.20} and \eqref{eq:2.21},
$$
e(a)\le \limsup_{R\rightarrow\infty}E_a(U_R)=-\infty.
$$
It implies  in this case that $e(a)$ has no minimizers.

If $a=a^*$,  we know from \eqref{eq:2.20} and \eqref{eq:2.21} that
$$e(a^*)\le V(x_0)$$
for all $x_0\in \mathbb{R}^3$. It results
$$
e(a^*)\le \inf_{x\in\mathbb{R}^3}V(x)=0.
$$
By \eqref{eq:2.3}, $e(a^*)\geq 0$, and then  $e(a^*)=0$.

We now show that there are no minimizers of $e(a)$ if $a=a^*$. Suppose, by contradiction, that $u$ is a minimizer of $e(a^*)$, so is $|u|$ by Lemma \ref{lem:A2}. From \eqref{eq:2.3} we obtain
\begin{equation}\label{eq:2.22}
\int_{\mathbb{R}^N}Vu^2\,dx=0,
\end{equation}
which implies
\[
\int_{\mathbb{R}^N}|(-\Delta)^{\frac{1}{4}}|u||^2\,dx=\frac{\|Q\|_2^2}{2}
\int_{\mathbb{R}^N}(|x|^{-1}*u^2)u^2\,dx.
\]

Taking into account the relation in \eqref{eq:1.9}, we know that $|u|$ is also a minimizer of the Weinstein functional $I(u)$ defined in \eqref{eq:1.8a}. After a suitable rescaling
$|u(x)|\rightarrow a|u(bx)|$ for some $a>0$ and $b>0$, $|au(bx)|$ is a nonnegative solution of \eqref{eq:1.7}.
 By Lemma 2.2 in \cite{FL1}, we obtain
\[
|u|\ge C(1+|x|)^{-4},
\]
which is a contradiction to \eqref{eq:2.22}. The assertion follows.
 \qquad$\Box$

 \bigskip

\section{Asymptotic behavior of minimizers}

\bigskip

In this section, we study asymptotic behavior of minimizers $u_a$ as $a\to a^*$. We will see that the $H^{\frac 12}(\mathbb{R}^N)$ norm of $u_a$ tends to infinity if $a\to a^*$. Then, using the blow-up argument, we find the  correct shape of the limit function.

\begin{Lemma} \label{lem:3.1}
If  $u_a$ is a minimizer of $e(a)$, then $\|(-\Delta)^{\frac 14} u_a\|_2\to \infty$
as $a\to a^*$.
\end{Lemma}

\begin{proof}
We argue indirectly. If the assertion is not true, there would exist a positive constant $C>0$
such that for $\delta>0$ small and $a^*-\delta\leq a\leq a^*$,
\[
\|(-\Delta)^{\frac 14} u_a\|_2\leq C\quad {\rm and}\quad e(a)=E_{a}(u_a).
\]
Hence, there exists a bounded sequence $\{u_{a_k}\}$
in $H^{\frac{1}{2}}(\mathbb{R}^3)$ with $a_k\to a^*$ as $k\to\infty$  so that $e(a_k)=E_{a_k}(u_k)$.

We note from \eqref{eq:2.3} that $e(a)$ is decreasing in $a\in [0,a^*)$ and satisfies $0\le e(a)\le e(0)$.
The boundedness of $e(a)$ for $a\in[0,a^*]$ and  \eqref{eq:2.3} yield
\[
\int_{\mathbb{R}^3}Vu_{a_k}^2\,dx\le C,
\]
which implies that $\{u_{a_k}\}$
is bounded in $\mathcal{H}$. So we may assume $u_{a_k}\rightharpoonup u$ in $\mathcal{H}$, and by Lemma \ref{lem:2.1}, $$
u_{a_k}\to u\quad {\rm in} \quad L^q(\mathbb{R}^3)
$$
as $k\to\infty$ for $2\le q< 3$. Hence, Lemma \ref{lem:2.2} yields
  $$
  \int_{\mathbb{R}^3}(|x|^{-1}*u_{a_k}^2)u_{a_k}^2\,dx\to
  \int_{\mathbb{R}^3}(|x|^{-1}*u^2)u^2\,dx
  $$
as $k\to\infty$.

 By \eqref{eq:2.3},  $e(a_k)\ge 0$. Choosing $a=a_k$ in \eqref{eq:2.20} and $x_0$ being a zero point of $V(x)$, and letting $k\to\infty$ then $R\to\infty$ in \eqref{eq:2.20} and \eqref{eq:2.21}, we deduce that
\[
\limsup_{k\rightarrow \infty}e(a_k)\leq\limsup_{k\rightarrow \infty}E_{a_k}(U_R)\le 0.
\]
Thus, $\lim_{k\rightarrow \infty}e(a_k)=0.$ Consequently,
\[
 0=e(a^*)\le E_{a^*}(u)\le \lim_{k\rightarrow \infty}E_{a_k}(u_{a_k})
 =\lim_{k\rightarrow \infty}e(a_k)=0.
\]
 Therefore, $u$ is the minimizer of $e(a^*)$, a contradiction to $(ii)$ of Theorem \ref{thm:1.1}.

\end{proof}

In order to prove Theorem \ref{thm:1.2}, we commence with the following lemmas.

\begin{Lemma} \label{lem:3.2} Suppose that $V$ satisfies \eqref{eq:1.11}. In the case
either (i) $m\neq0$ and $0<p<1$;
or (ii) $m=0$ and $0<p<\frac 52$; or (iii) $m=0$ and $0<\sum_{i=1}^np_i<5$;
there exist
$M_1>M_2>0$ such that
\begin{equation}\label{eq:3.1}
M_2(a^*-a)^{\frac{p}{p+1}}\le e(a)\le M_1(a^*-a)^{\frac{p}{p+1}}.
\end{equation}
\end{Lemma}

\begin{proof} Since $e(a)$ is decreasing and bounded for $a\in [0,a^*]$, it suffices to prove the result for $a$ close to $a^*$. By \eqref{eq:2.3}, for any $\gamma>0$ and $u\in\mathcal{H}$ with $\|u\|_2^2=1$, there holds
\begin{equation}\label{eq:3.2}
E_a(u)\ge\Big(1-\frac{a}{a^*}\Big)\int_{\mathbb{R}^3}|(-\Delta)^{\frac{1}{4}}u|^2\,dx
 +\gamma-\int_{\mathbb{R}^3}(\gamma-V)_+u^2\,dx.
\end{equation}
For $\varepsilon>0$, by the Young inequality, we have
\[
\int_{\mathbb{R}^3}(\gamma-V)_+u^2\,dx
\le \frac{1}{4}\varepsilon^{-4}\int_{\mathbb{R}^3}(\gamma-V)_+^4\,dx+
\frac{3}{4}\varepsilon^{\frac{4}{3}}\int_{\mathbb{R}^3}|u|^{\frac{8}{3}}\,dx.
\]
A special case of  the Gagliardo-Nirenberg inequality \eqref{eq:1.7a}
$$
\int_{\mathbb{R}^3}|u|^{\frac{8}{3}}\,dx\le C_1\int_{\mathbb{R}^3}|(-\Delta)^{\frac{1}{4}}u|^2\,dx
\Big(\int_{\mathbb{R}^3}u^2\,dx\Big)^{\frac{1}{3}}
$$
yields that
\[
\int_{\mathbb{R}^3}(\gamma-V)_+u^2\,dx\le
\frac{1}{4}\varepsilon^{-4}\int_{\mathbb{R}^3}(\gamma-V)_+^4\,dx+ \frac{3}{4}\varepsilon^{\frac{4}{3}}C_1\int_{\mathbb{R}^3}|
(-\Delta)^{\frac{1}{4}}u|^2\,dx.
\]

Now, choose $l\in(0,1)$ so that the $n$ balls
$$\{x||x-x_i|^{p_i}\le l\}\,( i=1,2,...,n)
$$
are mutually disjoint.
 Denote $$K=\max_i\, l^{\frac{1-n}{p_i}}.$$
 Choosing $\gamma>0$ such that
 $\gamma<\min\{l^n, \big(\frac{l}{K}\big)^p\}$, we have
 \begin{equation}\label{eq:3.3a}
\{x\in\mathbb{R}^3:\,V(x)\le \gamma\}\subseteq \cup_{i=1}^n\,\{x\in\mathbb{R}^3:\,|x-x_i|\le K\gamma^{\frac1p}\}.
\end{equation}
Indeed, if $x\not\in \cup_{i=1}^n\big\{x\big||x-x_i|^{p_i}\le l\big\}$, then
$$V(x)=\Pi_{i=1}^n|x-x_i|^{p_i}\ge l^n>\gamma.$$
 While if $x\in \{x\in \mathbb{R}^3||x-x_i|^{p_i}\le l\}$, the fact
 $$
 V(x)=\Pi_{i\neq j}|x-x_j|^{p_j}|x-x_i|^{p_i}\le \gamma
 $$
 implies that
 $$l^{n-1}|x-x_i|^{p_i}\le \gamma,$$
 that is
 $$|x-x_i|\le l^{\frac{1-n}{p_i}}\gamma^{\frac{1}{p_i}}\le K\gamma^{\frac{1}{p}}.$$
Thus, for $\gamma>0$ small enough, we get \eqref{eq:3.3a}.

On the other hand, if $x\in \{x||x-x_i|^{p_i}\le K\gamma^{\frac1p}\}$, then
$x\in \{x||x-x_i|^{p_i}\le l\}$ and
\[
 V(x)\ge l^{n-1}|x-x_i|^{p_i}\ge l^{n-1}|x-x_i|^{p}\ge \Big(\frac{|x-x_i|}{K}\Big)^p.
\]
 Hence,
\begin{align}\label{eq:3.3}
&\int_{\mathbb{R}^3}(\gamma-V)_+^4\,dx\nonumber\\
\leq&\sum_{i=1}^n\int_{\{|x-x_i|\le K\gamma^{\frac{1}{p}}\}}
\Big(\gamma-\Big(\frac{|x-x_i|}{K}\Big)^p\Big)^4\,dx\nonumber\\
=&nK^3\gamma^{4+\frac{3}{p}}\int_{\{|x|\le 1\}}(1-|x|^p)^4\,dx
\nonumber\\
=&C_2\gamma^{4+\frac{3}{p}}.
\end{align}
It follows from \eqref{eq:3.2} and \eqref{eq:3.3} that
\[
\begin{split}
&E_a(u)\ge\Big(1-\frac{a}{a^*}\Big)\int_{\mathbb{R}^3}|(-\Delta)^{\frac{1}{4}}u|^2\,dx
 +\gamma\\
&-\frac{3}{4}\varepsilon^{\frac{4}{3}}C_1\int_{\mathbb{R}^3}|(-\Delta)^{\frac{1}{4}}u|^2\,dx-
\frac{1}{4}\varepsilon^{-4}C_2\gamma^{4+\frac{3}{p}}.\\
\end{split}
\]
Choosing $\varepsilon>0$ such that $\frac{3}{4}\varepsilon^{\frac{4}{3}}C_1=1-\frac{a}{a^*}$, we obtain
\[
E_a(u)\ge \gamma-C_3(a^*-a)^{-3}\gamma^{4+\frac{3}{p}}.
\]
Then, taking $\gamma=\big(\frac{p(a^*-a)^3}{(4p+3)C_3}\big)^{\frac{p}{3(p+1)}}$, we see that the inequality
$$
C_2(a^*-a)^{\frac{p}{p+1}}\le e(a)
$$
is valid.

Now we turn to establish the upper bound for $e(a)$. The cases $m = 0 $ and $m\not =0$ are discussed separately. For the case $m=0$,  we distinguish two cases: $(i)\,\sum_{i=1}^np_i<5$, and  $(ii)\,0< p<\frac52$.

In the case $(i)\, \sum_{i=1}^np_i<5$, we choose
$$u_\tau=\frac{\tau^{\frac{3}{2}}}{\|Q\|_2}Q(\tau(x-x_j))$$
with $\tau>0$ and $x_j\in \mathcal{Z}$. In the same way as the proof of \eqref{eq:2.12}, we derive that
\begin{align}\label{eq:3.4}
&\int_{\mathbb{R}^3}u_{\tau}\sqrt{-\Delta+m^2}u_{\tau}\,dx\nonumber\\
&\le\frac{\tau}{\|Q\|_2^2}\int_{\mathbb{R}^3}Q\sqrt{-\Delta}Q\,dx+
\frac{m^2\tau^{-1}}{2\|Q\|_2^2}\int_{\mathbb{R}^3}Q(-\Delta)^{-\frac{1}{2}}Q\,dx
\nonumber\\
&=\frac{\tau}{\|Q\|_2^2}\int_{\mathbb{R}^3}|(-\Delta)^{\frac{1}{4}}Q|^2\,dx+ Cm^2\tau^{-1}
\end{align}
as $\tau\to\infty$, where $C>0$ is a constant depending on $\|Q\|^2_{H^{-\frac 12}(\mathbb{R}^N)}$.
A direct computation yields that
\begin{equation}\label{eq:3.5}
\int_{\mathbb{R}^3}(|x|^{-1}*u_{\tau}^2)u_{\tau}^2\,dx
=\frac{\tau}{\|Q\|_2^4}\int_{\mathbb{R}^3}(|x|^{-1}*Q^2)Q^2\,dx.
\end{equation}
For $\tau>1$, by \eqref{eq:1.10} and the Young inequality,
\begin{equation}\label{eq:3.6}
\big|\tau^pV(\tau^{-1}y+x_j)Q(y)^2\big|
\le C\Big(|y|^{\sum_{i=1}^{n}p_i}+1)(1+|y|)^{-8}.
\end{equation}
Since $\sum_{i=1}^np_i<5$, the function on the right hand side of \eqref{eq:3.6} belongs to $L^1(\mathbb{R}^N)$. So by the Lebesgue dominated convergence theorem, \eqref{eq:1.12} and \eqref{eq:1.13}, we derive
\begin{align}\label{eq:3.7}
\int_{\mathbb{R}^3}Vu_\tau^2\,dx
&=\|Q\|_2^{-2}\int_{\mathbb{R}^3}V(\tau^{-1}y+x_j)Q^2(y)\,dy\nonumber\\
&=\|Q\|_2^{-2}\tau^{-p}\Big(\kappa\int_{\mathbb{R}^3}|y|^pQ^2(y)\,dy+o(1)\Big)
\end{align}
as $\tau\to \infty$.

The Pohozaev identity, see Lemma 5 in \cite{LenLew},
\begin{equation}\label{eq:3.8}
2\int_{\mathbb{R}^3}|(-\Delta)^{\frac{1}{4}}Q|^2\,dx=\int_{\mathbb{R}^3}(|x|^{-1}*Q^2)Q^2\,dx
=2\int_{\mathbb{R}^3}Q^2\,dx,
\end{equation}
allows us to deduce from \eqref{eq:3.4}-\eqref{eq:3.7} and $m=0$ that
\[
\begin{split}
E_a(u_\tau)&\le\frac{\tau}{\|Q\|_2^2}\Big(\int_{\mathbb{R}^3}|(-\Delta)^{\frac{1}{4}} Q|^2\,dx-\frac{a}{2\|Q\|_2^2}\int_{\mathbb{R}^3}(|x|^{-1}*Q^2)Q^2\,dx
\Big)\\
&+\frac{\tau^{-p}\kappa}{\|Q\|_2^2}\int_{\mathbb{R}^3}|y|^pQ^2(y)\,dy
+o(\tau^{-p})\\
&=\frac{(a^*-a)}{a^*}\tau
+\frac{\tau^{-p}\kappa}{a^*}\int_{\mathbb{R}^3}|y|^pQ^2(y)\,dy
+o(\tau^{-p}).
\end{split}
\]

Choosing $\tau=\mu(a^*-a)^{-\frac{1}{p+1}}$ with $\mu$ defined in \eqref{eq:1.14}, we obtain the upper bound for $a$ close to $a^*$. Precisely, we have
\begin{equation}\label{eq:3.9}
\limsup_{a\rightarrow a^*}\frac{e(a)}{(a^*-a)^{\frac{p}{p+1}}}\le \limsup_{\tau\rightarrow \infty}
\frac{E_a(u_\tau)}{(a^*-a)^{\frac{p}{p+1}}}
\le\frac{(p+1)\mu}{pa^*}.
\end{equation}

Next, we consider the case $(ii)\, 0<p<\frac52$.

Let $U_R$ be defined in \eqref{eq:2.4} with $x_0\in \mathcal{Z}$. For $\delta>0$ small enough, we have
\[
V(x)\le C|x-x_0|^p
\]
if $|x-x_0|\le 2\delta.$ Hence,
\[
R^pV\big(R^{-1}x+x_0\big)\varphi^2(R^{-1}x)Q^2(x)\le C|x|^pQ^2(x),
\]
and $|x|^pQ^2(x)$ belongs to $L^1(\mathbb{R}^N)$ since $p<\frac{5}{2}$. Similarly, by the Lebesgue dominated convergence theorem, \eqref{eq:1.12} and \eqref{eq:1.13}, there holds
\begin{align}\label{eq:3.11}
\int_{\mathbb{R}^3}VU_R^2\,dx&=\frac{A_R^2}{\|Q\|_2^2}
\int_{\mathbb{R}^3}V(R^{-1}x+x_0)\varphi(R^{-1}x)^2Q^2(x)\,dx
\nonumber\\
&=\frac{R^{-p}\kappa}{\|Q\|_2^2}\Big(\int_{\mathbb{R}^3}|x|^pQ(x)^2dx+o(1)\Big)
\end{align}
as $R\to \infty.$ It follows from \eqref{eq:2.18}, \eqref{eq:3.11} and the Pohozaev identity \eqref{eq:3.8}
that
\[
\begin{split}
&E_a(U_R)\\
\le&\frac{R}{\|Q\|_2^2}\Big(\int_{\mathbb{R}^3}|(-\Delta)^{\frac{1}{4}} Q|^2\,dx-\frac{a}{2\|Q\|_2^2}\int_{\mathbb{R}^3}(|x|^{-1}*Q^2)Q^2\,dx
\Big)\\
+&\frac{R^{-p}\kappa}{\|Q\|_2^2}\int_{\mathbb{R}^3}|y|^pQ^2(y)\,dy
+o(R^{-p})+O(R^{-\frac{5}{2}})\\
=&\frac{(a^*-a)}{a^*}R
+\frac{R^{-p}\kappa}{a^*}\int_{\mathbb{R}^3}|y|^pQ^2(y)\,dy
+o(R^{-p})+O(R^{-\frac{5}{2}}).
\end{split}
\]
In the same way,  we obtain the upper bound \eqref{eq:3.9}.

 For $m\neq0$ and $0<p<1$, a similar process for the case $m=0$ can be used to derive the upper bound \eqref{eq:3.9}.
 The proof is complete.
\end{proof}

\bigskip


\bigskip

Now, we estimate the nonlocal term.

\begin{Lemma}\label{lem:3.3}
 Under the same conditions of Lemma \ref{lem:3.2}, there exists $K>0$, independent of $a$, such that
\begin{align}\label{eq:3.13}
0<K(a^*-a)^{-\frac{1}{p+1}}\le \int_{\mathbb{R}^3}(|x|^{-1}*u_a^2)u_a^2dx
\le \frac{1}{K}(a^*-a)^{-\frac{1}{p+1}}.
\end{align}\end{Lemma}

\begin{proof}  Since $\sqrt{-\Delta+m^2}\ge\sqrt{-\Delta}$, by \eqref{eq:1.8},
\begin{align*}
e(a)=E_a(u_a)&\ge \int_{\mathbb{R}^3}\big(|(-\Delta)^{\frac{1}{4}}u_a|^2+Vu_a^2\big)\,dx
 -\frac{a}{2}\int_{\mathbb{R}^3}\big(|x|^{-1}*u_a^2\big)u_a^2\,dx\nonumber\\
  &\ge \frac{a^*-a}{2}\int_{\mathbb{R}^3}\big(|x|^{-1}*u_a^2\big)u_a^2\,dx+\int_{\mathbb{R}^3}Vu_a^2\,dx\nonumber\\
 &\ge \frac{a^*-a}{2}\int_{\mathbb{R}^3}\big(|x|^{-1}*u_a^2\big)u_a^2\,dx,
 \end{align*}
 which together with \eqref{eq:3.1} implies
 \[
 \int_{\mathbb{R}^3}\big(|x|^{-1}*u_a^2\big)u_a^2\,dx\le 2M_1(a^*-a)^{-\frac{1}{p+1}}.
 \]

 To get the lower bound, set $\theta>0$ such that $M_2(1+\theta)^{\frac{p}{p+1}}=2M_1$. Since
\[
e(a)= E_a(u_a)\le E_a(u_0),
\]
that is
\[
\begin{split}
&\int_{\mathbb{R}^3}(u_a\sqrt{-\Delta+m^2}u_a+Vu_a^2)\,dx-\frac{a}{2}
\int_{\mathbb{R}^3}(|x|^{-1}*u_a^2)u_a^2\,dx\\
&\le \int_{\mathbb{R}^3}(u_0\sqrt{-\Delta+m^2}u_0+Vu_0^2)dx-\frac{a}{2}
\int_{\mathbb{R}^3}(|x|^{-1}*u_0^2)u_0^2\,dx\\
\end{split}
\]
 and
\[
 e(0)=E_0(u_0)\le E_0(u_a),
\]
 namely,
 \[
  \int_{\mathbb{R}^3}(u_0\sqrt{-\Delta+m^2}u_0+Vu_0^2)\,dx
 \le \int_{\mathbb{R}^3}(u_a\sqrt{-\Delta+m^2}u_a+Vu_a^2)\,dx,
\]
 we obtain for $a\in [0,\frac{\theta}{1+\theta}a^*)$ that
\[
\begin{split}
 &\int_{\mathbb{R}^3}\big(|x|^{-1}*u_a^2\big)u_a^2\,dx\\
 \ge& \int_{\mathbb{R}^3}\big(|x|^{-1}*u_0^2\big)u_0^2\,dx\\
 \ge& \Big(\frac{a^*}{1+\theta}\Big)^{\frac{1}{p+1}}(a^*-a)^{-\frac{1}{p+1}}\int_{\mathbb{R}^3}
 \big(|x|^{-1}*u_0^2\big)u_0^2\,dx.
 \end{split}
\]
 If $a\in[\frac{\theta}{1+\theta}a^*,a^*)$, let $b=a-\theta(a^*-a)\in [0,a)$. Then, the  fact
\[
\begin{split}
 e(b)\le E_b(u_a)&=E_a(u_a)+\frac{a-b}{2}\int_{\mathbb{R}^3}\big(|x|^{-1}*u_a^2\big)
 u_a^2\,dx\\
 &=e(a)+\frac{a-b}{2}\int_{\mathbb{R}^3}\big(|x|^{-1}*u_a^2\big)
 u_a^2\,dx
 \end{split}
\]
 and \eqref{eq:3.1} yield
\[
\begin{split}
 &\frac{1}{2}\int_{\mathbb{R}^3}\big(|x|^{-1}*u_a^2\big)
 u_a^2\,dx\nonumber\\
 &\ge\frac{e(b)-e(a)}{a-b}\nonumber\\
 &\ge \frac{M_2(a^*-b)^{\frac{p}{p+1}}-M_1(a^*-a)^{\frac{p}{p+1}}}{a-b}\nonumber\\
 &=\frac{M_2(1+\theta)^{\frac{p}{p+1}}- M_1}{\theta}(a^*-a)^{-\frac{1}{p+1}}
 =\frac{M_1}{\theta}(a^*-a)^{-\frac{1}{p+1}}.
  \end{split}
\]
The assertion follows.
\end{proof}

\bigskip

Let $u_a$ be a non-negative minimizer of \eqref{eq:1.5} and $\lambda_a=(a^*-a)^{\frac{1}{p+1}}.$
We now use the blow-up argument to analyze the collapse concentration of the minimizer $u_a$ as $a\to a^*$.
By \eqref{eq:2.3},
\[
\int_{\mathbb{R}^3}|(-\Delta)^{\frac{1}{4}}u_a|^2\,dx\le \frac{a^*}{a^*-a}E_a(u_a)
=\frac{a^*}{a^*-a}e(a)
\]
and
\[
\int_{\mathbb{R}^3}Vu_a^2\,dx\le E_a(u_a)=e(a).
\]
Then,  Lemma \ref{lem:3.2} gives
\begin{equation*}
\int_{\mathbb{R}^3}|(-\Delta)^{\frac{1}{4}}u_a|^2\,dx
\le \frac{a^*}{a^*-a}M_1(a^*-a)^{\frac{p}{p+1}}
\le C\lambda_a^{-1}
\end{equation*}
and
\[
\int_{\mathbb{R}^3}Vu_a^2\,dx\le C(a^*-a)^{\frac{p}{p+1}}= C\lambda_a^p.
\]
For $1\le i\le n$, we define
\begin{equation}\label{eq:3.16}
w_a^{i}(x)=\lambda_a^{\frac{3}{2}}u_a(\lambda_a x+x_i)
\end{equation}
with $\|w_a^{i}\|_2=\|u_a\|_2=1$. It is readily to see that
\begin{equation}\label{eq:3.17}
\int_{\mathbb{R}^3}|(-\Delta)^{\frac{1}{4}}w_a^{i}|^2\,dx=\lambda_a\int_{\mathbb{R}^3}\big|(-\Delta)^{\frac{1}{4}}
u_a\big|^2\,dx\le C,
\end{equation}
and
\begin{equation}\label{eq:3.18}
\int_{\mathbb{R}^3}V(\lambda_a x+x_i)w_a^{i}(x)^2\,dx\le C\lambda_a^p.
\end{equation}
For $\gamma>0$,
$$
\int_{\{x\in \mathbb{R}^3: V(x)\ge \gamma\lambda_a^p\}}u_a^2\,dx\le \frac{1}{\gamma\lambda_a^p}
\int_{\mathbb{R}^3}Vu_a^2\,dx\le\frac{C}{\gamma},
$$
which implies that
$$
\int_{\{x\in \mathbb{R}^3: V(x)\le \gamma\lambda_a^p\}}u_a^2\,dx=1-\int_{\{x\in \mathbb{R}^3: V(x)\ge \gamma\lambda_a^p\}}u_a^2\,dx\ge 1-\frac{C}{\gamma}.
$$
If $\lambda_a>0$ small, as in the proof of \eqref{eq:3.3a}, we have
$$
\{x\in \mathbb{R}^3: V(x)\le \gamma \lambda_a^p\}\subset \cup_{i=1}^n\{x\in \mathbb{R}^3:|x-x_i|\le C\gamma^{\frac{1}{p}}\lambda_a\},
$$
where $\{x\in \mathbb{R}^3: |x-x_i|\le C\gamma^{\frac{1}{p}}\lambda_a\}$ are mutually disjoint. Therefore,
\[
\begin{split}
&\int_{\{x\in \mathbb{R}^3: V(x)\le \gamma\lambda_a^p\}}u^2_a(x)\,dx\\
\le &\sum_{i=1}^{n}\int_{\{x\in \mathbb{R}^3: |x-x_i|\le C\gamma^{\frac{1}{p}}\lambda_a\}}u^2_a(x)\,dx\\
=&\sum_{i=1}^{n}\int_{\{x\in \mathbb{R}^3:|x|\le C\gamma^{\frac{1}{p}}\}}w_a^{i}(x)^2\,dx\le1,\\
\end{split}
\]
which implies
\begin{equation}\label{eq:3.19}
 1-\frac{C}{\gamma}\le\sum_{i=1}^{n}\int_{\{|x|\le C\gamma^{\frac{1}{p}}\}}w_a^{i}(x)^2\,dx\le1,
\end{equation}
for $a$ close to $a^*$.

By \eqref{eq:3.17}, $\{w_a^{i}\}$ is bounded  in $H^{\frac{1}{2}}(\mathbb{R}^3)$. Assuming $a_k\to a^*$ as $k\to\infty$, correspondingly, we have
\[
w_{a_k}^{i}\rightharpoonup w_0^{i}\quad{\rm in}\quad H^{\frac{1}{2}}(\mathbb{R}^3),\quad w_{a_k}^{i}\to w_0^{i}\quad{\rm in}\quad L^q_{loc}(\mathbb{R}^3)
\]
as $k\to\infty$ for $2\le q<3$.
Replacing $w_a^{i}$ in \eqref{eq:3.19} by $w_{a_k}^{i}$, and letting $k\to\infty$, then $\gamma\to\infty$, we obtain
\[
\sum_{i=1}^{n}\int_{\mathbb{R}^3}w_0^{i}(x)^2\,dx=1.
\]
Then, there exists $j$ such that $w_0^{j}\ge0$ and  $w_0^{j}\not\equiv0$. We claim that
\[
w_{a_k}^{j}\to w_0^{j}\quad{\rm in}\quad L^2(\mathbb{R}^3)
\]
as $k\to\infty$. This is the case if we may show $\|w_0^{j}\|_2^2=1$. So it suffices to prove the following result.
\begin{Lemma}\label{lem:3.4}
There exist $\beta>0$ and $y_0\in \mathbb{R}^3$,
and a radially decreasing and positive ground state $Q$ of \eqref{eq:1.7} such that
\[
w_0^{j}=\frac{\beta^{3}}{\|Q\|_2}Q(\beta^2(x-y_0)),
\]
and then $\|w_0^{j}\|_2^2=1$. Moreover,
\[
w_{a_k}^{j}\to w_0^{j}\quad{\rm in}\quad L^2(\mathbb{R}^3)
\]
as $k\to\infty$.
\end{Lemma}

\begin{proof} We know that the minimizer $u_a$ satisfies
$$\sqrt{-\Delta+m^2}u_a+Vu_a-a(|x|^{-1}*u_a^2)u_a=\mu_au_a,$$
where $\mu_a$ is the Lagrange multiplier.
Hence, we have
\[
\mu_a=e(a)-\frac{a}{2}\int_{\mathbb{R}^3}(|x|^{-1}*u_a^2)u_a^2\,dx.
 \]
By Lemmas \ref{lem:3.2} and \ref{lem:3.3}, $\lambda_a\mu_a$ is negative and bounded for $a$ close to $a^*$. So we may assume
that there exist $\beta>0$ and a sequence $\{a_k\}$ such that
$$
\lambda_{a_k}\mu_{a_k}\rightarrow -\beta^2
$$
as $k\to \infty$.
By the definition of $\sqrt{-\Delta+m^2}$, we deduce
\[
\begin{split}
&\lambda_a^{\frac{3}{2}}\sqrt{-\Delta+m^2}u_a(\lambda_ax+x_j)\\
=&\lambda_a^{\frac{3}{2}}\mathcal{F}^{-1}\big(\sqrt{\xi^2+m^2}\mathcal{F}(u_a)\big)(\lambda_ax+x_j)\\
=&(2\pi)^{-3}\lambda_a^{\frac{3}{2}}\int_{\mathbb{R}^3}{\rm e}^{i(\lambda_ax+x_j)\cdot\xi}\sqrt{\xi^2+m^2}\,d\xi
\int_{\mathbb{R}^3}{\rm e}^{-i\xi\cdot y}u_a(y)\,dy\\
=&(2\pi)^{-3}\int_{\mathbb{R}^3}{\rm e}^{ix\cdot\xi}\sqrt{\lambda_a^{-2}\xi^2+m^2}\,d\xi
\int_{\mathbb{R}^3}{\rm e}^{-i\xi\cdot y}w_a^{j}(y)\,dy\\
=&\mathcal{F}^{-1}\Big(\sqrt{\lambda_a^{-2}\xi^2+m^2}F(w_a^{j})\Big)(x)\\
=&\lambda_a^{-1}\sqrt{-\Delta+\lambda_a^2m^2}w_a^{j}(x).
\end{split}
\]
In the same way, we have
\[
\begin{split}
&a\lambda_a^{\frac{3}{2}}\Big(\big(|x|^{-1}*u_a^2\big)u_a\Big)(\lambda_ax+x_j)\\
=&a\lambda_a^{\frac{3}{2}}u_a(\lambda_ax+x_j)\int_{\mathbb{R}^3}\frac{1}{|\lambda_ax+x_j-y|}u_a^2(y)\,dy\\
=&a\lambda_a^{-1}w_a^{j}(x)\int_{\mathbb{R}^3}\frac{1}{|x-y|}(w_a^{j})^2(y)\,dy \\
=&a\lambda_a^{-1}\Big(|x|^{-1}*\big(w_a^{j}\big)^2
\Big)w_a^{j}(x).
\end{split}
\]
Hence, the function  $w_a^{j}$ satisfies
\begin{align}
&\sqrt{-\Delta+\lambda_a^2m^2}w_a^{j}+\lambda_a V(\lambda_a x+x_j)w_a^{j}\nonumber\\
= &a\Big(|x|^{-1}*
\big(w_a^{j}\big)^2\Big)w_a^{j}+\lambda_a\mu_aw_a^{j}.\nonumber
\end{align}

For $\varphi\in C_c(\mathbb{R}^3)$, we first estimate
\begin{align}\label{eq:3.23}
&\quad\int_{\mathbb{R}^3}\big(\sqrt{-\Delta+\lambda_{a_k}^2m^2}w_{a_k}^{j}-
\sqrt{-\Delta}w_0\big)\varphi\,dx
\nonumber\\
&=\int_{\mathbb{R}^3}\sqrt{\xi^2+\lambda_{a_k}^2m^2}
F\big(w_{a_k}^{j}\big)\mathcal{F}(\varphi) -\int_{\mathbb{R}^3}|\xi|\mathcal{F}(w_0^{j})\big)\mathcal{F}(\varphi)\,d\xi
\nonumber\\
&=\int_{\mathbb{R}^3}\Big(\sqrt{\xi^2+\lambda_{a_k}^2m^2}-|\xi|\Big)\mathcal{F}\big(w_{a_k}^{j}\big)\mathcal{F}(\varphi)\,d\xi
+\int_{\mathbb{R}^3}|\xi|\mathcal{F}\Big(w_{a_k}^{j}-w_0^{j}\Big)F(\varphi)\,d\xi
\nonumber\\
&= I +II.
\end{align}

We write
\[
\begin{split}
II
&= \int_{\mathbb{R}^3}\mathcal{F}^{-1}\big(\sqrt{|\xi|}\mathcal{F}\big(w_{a_k}^{j}-
w_0^{j}\big)\Big)\mathcal{F}^{-1}\big(\sqrt{|\xi|}\mathcal{F}(\varphi)\big)\,dx\\
&= \int_{\mathbb{R}^3}(-\Delta)^{\frac{1}{4}}
\big(w_{a_k}^{j}-w_0^{j}\big)(-\Delta)^{\frac{1}{4}}\varphi\,dx.\\
\end{split}
\]
The weak convergence $w_{a_k}^{j}\rightharpoonup w_0^{j} $ in $H^{\frac{1}{2}}(\mathbb{R}^3)$ yields
\begin{align}\label{eq:3.24}
\int_{\mathbb{R}^3}(-\Delta)^{\frac{1}{4}}
\big(w_{a_k}^{j}-w_0^{j}\big)(-\Delta)^{\frac{1}{4}}\varphi\, dx\rightarrow 0
\end{align}
as $k\rightarrow \infty$. On the other hand,
\[
I =\int_{\mathbb{R}^3}\frac{\lambda_{a_k}^2m^2}{\sqrt{\xi^2+\lambda_{a_k}^2m^2}+|\xi|}
\mathcal{F}(w_{a_k}^{j})\mathcal{F}(\varphi)\,d\xi.
\]
By the H\"{o}lder inequality and the Plancherel theorem, we infer that
\begin{align}\label{eq:3.25}
|I|&\le
|\lambda_{a_k}m|\int_{\mathbb{R}^3}|\mathcal{F}(w_{a_k}^{j})\mathcal{F}(\varphi)|\,dx
\nonumber\\
&\le |\lambda_{a_k}m|\|F(w_{a_k}^j)\|_2\|F(\varphi)\|_2
\nonumber\\
&=|\lambda_{a_k}m|\big\|w_{a_k}^{j}\big\|_2\|\varphi\|_2
\nonumber\\
&=|\lambda_{a_k}m|\|\varphi\|_2\rightarrow 0
\end{align}
as $k\to \infty$.
As a result of \eqref{eq:3.23}--\eqref{eq:3.25}, we have
\begin{equation}\label{eq:3.26}
\int_{\mathbb{R}^3}\big(\sqrt{-\Delta+\lambda_{a_k}^2m^2}w_{a_k}^{j}-
\sqrt{-\Delta}w_0^{j}\big)\varphi\, dx\rightarrow 0
\end{equation}
as $k\to \infty$.

Next, by \eqref{eq:3.18} and the H\"{o}lder inequality, we estimate
\begin{align*}
&\Big|\int_{\mathbb{R}^3}\lambda_{a_k} V\big(\lambda_{a_k} x+x_j\big)w_a^{j}\varphi \,dx\Big|
\nonumber\\
&\le
\lambda_{a_k}\Big(\int_{\mathbb{R}^3}V\big(\lambda_{a_k} x+x_j\big)w_{a_k}^{(j)}(x)^2dx\Big)^{\frac{1}{2}}
\Big(\int_{\mathbb{R}^3}V\big(\lambda_{a_k} x+x_j\big)\varphi(x)^2\,dx\Big)^{\frac{1}{2}}
\nonumber\\
&\le C\lambda_{a_k}^{\frac{p}{2}+1}\sup_{x\in supp\{\varphi\}}V(\lambda_{a_k} x+x_j)^{\frac{1}{2}}\|\varphi\|_2,
\end{align*}
which tends to zero as $k\to \infty$.

Finally, the Hardy-Littlewood-Soblev  inequality allows us to show

\begin{align}\label{eq:3.28}
\int_{\mathbb{R}^3}\big(|x|^{-1}*(w_{a_k}^{j})^2\big)\big(w_{a_k}^{j}-w_0^{j}\big)\varphi\, dx\le
C\big\|w_{a_k}\big\|_{\frac{12}{5}}^2\big\|\big(w_{a_k}^{j}-w_0^{j}\big)\varphi\big\|_{\frac{6}{5}},
\end{align}
and then it tends to zero as $k\to \infty$ since $w_{a_k}^{j}\to w_0^{j}$ in $L^q_{loc}(\mathbb{R}^3)$.
Now, choose a non-negative function $\Psi_R\in C_c(\mathbb{R}^3)$ such that
 $$\Psi_R\equiv1,\,x\in B_R;\,\,\Psi_R\equiv0,\,x\in \mathbb{R}^3\backslash B_{2R}$$
 with  $R>0$.
Arguing as \eqref{eq:3.28}, we find
 \begin{align}\label{eq:3.29}
 &\quad\Big|\int_{\mathbb{R}^3}|x|^{-1}*(w_0^{j}\varphi)((w_{a_k}^{j})^2
 -(w_0^{j})^2)
 \Psi_Rdx\Big|\nonumber\\
 &\le C\|w_0^{j}\varphi\|_{\frac{6}{5}}\Big\|\big((w_{a_k}^{j})^2-(w_0^{j})^2\big)\Psi_R\Big\|_{\frac{6}{5}}
 \nonumber\\
 &\le C\|w_0^{j}\varphi\|_{\frac{6}{5}}\big\|\big(w_{a_k}^{j}+w_0^{j}\big)\Psi_R\big\|_{\frac{12}{5}}
\big\|\big(w_{a_k}^{j}-w_0^{j}\big)\Psi_R\big\|_{\frac{12}{5}}
\nonumber\\
&\le C\|(w_{a_k}^{j}-w_0^{j}\|_{L^{\frac{12}{5}}(B_{2R})},
 \end{align}
it goes to zero as $k\to \infty$. By the Newton theorem, we obtain
 \begin{align}\label{eq:3.30}
 &\quad\Big|\int_{\mathbb{R}^3}|x|^{-1}*(w_0^{j}\varphi)\Big((w_{a_k}^{j})^2-(w_0^{j})^2\Big)
 (1-\Psi_R)\,dx\Big|
 \nonumber\\
 &\leq\int_{\mathbb{R}^3\setminus B_R}|x|^{-1}*\big|w_0^{j}\varphi\big|\Big|\Big((w_{a_k}^{j})^2-(w_0^{j})^2\Big)
 (1-\Psi_R)\Big|\,dx
 \nonumber\\
 &\le C\int_{\mathbb{R}^3\setminus B_R}|x|^{-1}\int_{\mathbb{R}^3}\big|w_0^{j}\varphi(y) \big|\,dy\Big|\Big((w_{a_k}^{j})^2-(w_0^{j})^2\Big)
 (1-\Psi_R)\Big|\,dx
 \nonumber\\
 &\le CR^{-1}\|w_0^{j}\|_2^2\|\varphi\|_2^2(\|w_{a_k}^{j}\|_2^2+\|w_0^{j}\|_2^2)
 \nonumber\\
 &\le CR^{-1}.
 \end{align}
By the Fubini theorem,
 \begin{align}\label{eq:3.31}
&\int_{\mathbb{R}^3}|x|^{-1}*((w_{a_k}^{j})^2-(w_0^{j})^2)w_0^{j}\varphi\,dx\nonumber\\
&=\int_{\mathbb{R}^3}|x|^{-1}*(w_0^{j}\varphi)\big((w_{a_k}^{j})^2-(w_0^{j})^2\big)\,dx
 \nonumber\\
 &=\int_{\mathbb{R}^3}|x|^{-1}*(w_0^{j}\varphi)\big((w_{a_k}^{j})^2-(w_0^{j})^2\big)
 \Psi_R\,dx\nonumber\\
 &+\int_{\mathbb{R}^3}|x|^{-1}*(w_0^{j}\varphi)\big((w_{a_k}^{j})^2-(w_0^{j})^2\big)
 (1-\Psi_R)\,dx.
 \end{align}
We deduce from \eqref{eq:3.29}--\eqref{eq:3.31} that
 $$\limsup_{k\rightarrow\infty}\Big|\int_{\mathbb{R}^3}|x|^{-1}*
 ((w_{a_k}^{j})^2-(w_0^{j})^2)w_0^{j}\varphi \,dx\Big|\le CR^{-1}.$$
Letting $R\to\infty$, we obtain
\begin{equation}\label{eq:3.32}\int_{\mathbb{R}^3}|x|^{-1}*
 ((w_{a_k}^{j})^2-(w_0^{j})^2)w_0^{j}\varphi\, dx\rightarrow0
 \end{equation}
as $k\to \infty$.
 It follows from \eqref{eq:3.28} and \eqref{eq:3.32} that
\[
\begin{split}
&\quad\int_{\mathbb{R}^3}\Big(\big(|x|^{-1}*(w_{a_k}^{j})^2\big)w_{a_k}^{j}-
\big(|x|^{-1}*(w_0^{j})^2\big)w_0^{j}\Big)\varphi\, dx\nonumber\\
&=\int_{\mathbb{R}^3}\big(|x|^{-1}*(w_{a_k}^{j})^2\big)
\big(w_{a_k}^{j}-w_0^{j}\big)\varphi\, dx\nonumber\\
&+\int_{\mathbb{R}^3}|x|^{-1}*\big((w_{a_k}^{j})^2-(w_0^{j})^2\big)w_0^{j}\varphi\, dx
\rightarrow0
\end{split}
\]
as $k\to \infty$.
Consequently,  $w_0^{j}$ satisfies
\begin{align*}
\sqrt{-\Delta}w_0^{j}-a^*(|x|^{-1}*(w_0^{j})^2)w_0^{j}=-\beta^2w_0^{j}.
\end{align*}

Let $w_0^{j}=\frac{\beta^3w(\beta^2(x-y_0))}{\|Q\|_2}$. Then $w$ satisfies \eqref{eq:1.7}. By Lemma \ref{lem:A3}, we have $\|w\|_2\ge\|Q\|_2$. On the other hand, due to the weak semi-continuity of the $L^2(\mathbb{R}^3)$ norm and $\|w_{a_k}^j\|_2^2=1$, we have $\|w_0^j\|_2^2\le1$, and then $\|w\|_2^2\le\|Q\|_2^2$. Thus, we get $\|w\|_2^2=\|Q\|_2^2$. Again by Lemma \ref{lem:A3}, $w$ is a ground state of \eqref{eq:1.7}. Thanks to Theorem 1.1 in \cite{FL1}, there exists a radial decreasing and positive ground state $Q$ of \eqref{eq:1.7} and  $y_0\in \mathbb{R}^3$ such that $w=Q(x-y_0).$
Thus, we conclude
$$w_0^{j}=\frac{\beta^3Q(\beta^2(x-y_0))}{\|Q\|_2},$$
and $\|w_0^{j}\|_2^2=1$. The proof is complete.

\end{proof}

Now, we are ready to prove Theorem \ref{thm:1.2}.

{\bf Proof of Theorem \ref{thm:1.2}}  We know by Lemma \ref{lem:3.4} that
$$
w_{a_k}^{j}(x) = \lambda_{a_k}^{\frac{3}{2}}u_{a_k}(\lambda_{a_k} x+x_j)
$$
defined in \eqref{eq:3.16} is bounded in $H^{\frac{1}{2}}(\mathbb{R}^3)$, and
\begin{equation}
w_{a_k}^{j}\rightarrow w_0^{j}\quad {\rm in}\quad L^2(\mathbb{R}^3)\nonumber
\end{equation}
as $k\to \infty$, where $w_0^{j}=\frac{\beta^3Q(\beta^2(x-y_0))}{\|Q\|_2}$. The proof of Theorem 1.2 will be completed once we determine $x_j$, $y_0$ and $\beta^2$. To this purpose, we will show that
\begin{align}\label{eq:3.34}
\liminf_{k\rightarrow \infty}\frac{e(a_k)}{(a^*-a_k)^{\frac{p}{p+1}}}
 \ge \frac{p+1}{p}\frac{\mu}{a^*}.
 \end{align}
Indeed, by \eqref{eq:3.17} and the Gagliardo-Nirenberg inequality,
\[
w_{a_k}^{j}\rightarrow w_0^{j}\quad {\rm in}\quad L^q(\mathbb{R}^3)
\]
as $k\to \infty$ for $2\le q<3$. It follows from Lemma \ref{lem:3.4} that
\begin{equation}\label{eq:3.27a}
\begin{split}
  e(a_k)&=E_{a_k}(u_{a_k})\\
  &\ge \frac{a^*-a_k}{2}\int_{\mathbb{R}^3}(|x|^{-1}*u_{a_k}^2)u_{a_k}^2\,dx
  +\int_{\mathbb{R}^3}Vu_{a_k}^2\,dx\\
  &=\frac{\lambda_{a_k}^p}{2}\int_{\mathbb{R}^3}\big(|x|^{-1}*(w_{a_k}^{j})^2\big)
  w_{a_k}^{j}(x)^2\,dx+\int_{\mathbb{R}^3}V(\lambda_{a_k}x+x_j)
  w_{a_k}^{j}(x)^2\,dx.\\
  \end{split}
\end{equation}
By  the Fatou lemma, we have
  \begin{align}\label{eq:3.35}
  &\liminf_{k\rightarrow\infty}\lambda_{a_k}^{-p}\int_{\mathbb{R}^3}V(\lambda_{a_k}x+x_j)
  w_{a_k}^{j}(x)^2\,dx\nonumber\\
  &=\liminf_{k\rightarrow\infty}\int_{\mathbb{R}^3}\int_{\mathbb{R}^3}
  \frac{V(\lambda_{a_k}x+x_j)}{|\lambda_{a_k}x|^p}|x|^pw_{a_k}^{j}(x)^2\,dx
  \nonumber\\
  &\ge \kappa \int_{\mathbb{R}^3}|x|^pw_0^{j}(x)^2\,dx,
  \end{align}
  where $\kappa$ is defined in \eqref{eq:1.13}. We write
\[
\int_{\mathbb{R}^3}|x|^pw_0^{j}(x)^2\,dx
  =\frac{1}{a^*\beta^{2p}}\int_{\mathbb{R}^3}|x+\beta^2y_0|^pQ^2(x)\,dx
 \]
and we claim that
\begin{equation}\label{eq:3.36}
\frac{1}{a^*\beta^{2p}}\int_{\mathbb{R}^3}|x+\beta^2y_0|^pQ^2(x)\,dx
\geq \frac{1}{a^*\beta^{2p}}\int_{\mathbb{R}^3}|x|^pQ^2(x)\,dx.
\end{equation}
The equality in \eqref{eq:3.36} holds if and only if $y_0=0$.
Indeed, since $Q$ is a radially symmetric and decreasing and
 \[
 \int_{\mathbb{R}^3}|x+\beta^2y_0|^pQ^2(x)\,dx
 =\int_{\mathbb{R}^3}|x-\beta^2y_0|^pQ^2(x)\,dx,
 \]
 we have
 \begin{equation*}
\begin{split}
&\quad\int_{\mathbb{R}^3}|x+\beta^2y_0|^pQ^2(x)\,dx-\int_{\mathbb{R}^3}|x|^pQ^2(x)\,dx\\
&=\frac12\int_{\mathbb{R}^3}(|x+\beta^2y_0|^p+|x-\beta^2y_0|^p-2|x|^p)Q^2(x)\,dx\\
&=\frac12\int_{\mathbb{R}^3}(|x+\beta^2y_0|^p-|x|^p)Q^2(x)\,dx+
\frac12\int_{\mathbb{R}^3}(|x-\beta^2y_0|^p-|x|^p)Q^2(x)\,dx\\
&=\frac12\int_{\mathbb{R}^3}(|x+\beta^2y_0|^p-|x|^p)Q^2(x)\,dx\\
&+\frac12\int_{\mathbb{R}^3}(|x|^p-|x+\beta^2y_0|^p)Q^2(x+\beta^2y_0)\,dx\\
&=\frac12\int_{\mathbb{R}^3}(|x+\beta^2y_0|^p-|x|^p)
\big(Q^2(x)-Q^2(x+\beta^2y_0)\big)\,dx\ge0.
\end{split}
 \end{equation*}

Therefore, by \eqref{eq:3.35}, \eqref{eq:3.36} and Lemmas \ref{lem:2.2} and \ref{lem:3.4} we find that
  \[
  \begin{split}
  &\liminf_{k \rightarrow \infty}\frac{e(a_k)}{(a^*-a_k)^{\frac{p}{p+1}}}\\
  \ge&\frac{1}{2}\int_{\mathbb{R}^3}\Big(|x|^{-1}*\big(w_0^{j}\big)^2\Big)w_0^{j}(x)^2\,dx
  +\frac{\kappa}{a^*\beta^{2p}}\int_{\mathbb{R}^3}|x|^pQ^2(x)\,dx\\
  =& \frac{\beta^2}{2(a^*)^2}\int_{\mathbb{R}^3}(|x|^{-1}*Q^2)Q^2\,dx+\frac{\kappa}{a^*\beta^{2p}}
  \int_{\mathbb{R}^3}|x|^pQ(x)^2\,dx\\
  \end{split}
  \]
Since
$$
\int_{\mathbb{R}^3}(|x|^{-1}*Q^2)Q^2\,dx=2\int_{\mathbb{R}^3}Q^2\,dx,
$$
we obtain
\begin{equation}
\liminf_{k \rightarrow \infty}\frac{e(a_k)}{(a^*-a_k)^{\frac{p}{p+1}}}
\geq\frac{1}{a^*}\Big(\beta^2+\frac{\kappa}{\beta^{2p}}\int_{\mathbb{R}^3}|x|^pQ(x)^2\,dx\Big).
 \nonumber\end{equation}
Take the minimum on the right hand side of the above inequality,
 it is achieved only at $\beta^2=\mu$, where $\mu$ is
defined in \eqref{eq:1.14}. Hence, \eqref{eq:3.34} holds true.

  Finally, we claim that  $x_j\in \mathcal{Z}$, $y_0=0$ and $\beta^2=\mu$. On the contrary, if any one of these cases does not happen, i.e. either $x_j\not\in\mathcal{Z}$, or  $y_0\neq0$, or $\beta^2\neq\mu$, it follows from \eqref{eq:3.34}, \eqref{eq:3.35}  and \eqref{eq:3.36} that
\[
\liminf_{k\rightarrow \infty}\frac{e(a_k)}{(a^*-a_k)^{\frac{p}{p+1}}}
 > \frac{p+1}{p}\frac{\mu}{a^*},
 \]
 which is a contradiction to \eqref{eq:3.9}. The proof of Theorem 2 is complete. $\Box$

 \bigskip

  \section{Appendix}

  \bigskip

In the appendix, we present some facts used in the sequel. For the completeness, we include proofs of these results.
 First, we have the following operator inequality.

  \begin{Lemma}\label{lem:A1} There holds
  \begin{equation*}
  \sqrt{-\Delta+R^{-2}m^2}\le \sqrt{-\Delta}+\frac{1}{2}R^{-2}m^2(-\Delta)^{-\frac{1}{2}}.
  \end{equation*}
  \end{Lemma}

\begin{proof} For $\varphi\in C_c^\infty(\mathbb{R}^3)$, by the Plancherel theorem, we have
\[
\begin{split}
&\quad\int_{\mathbb{R}^3}\varphi\sqrt{-\Delta+R^{-2}m^2}\,\varphi\, dx\\
&=\int_{\mathbb{R}^3}\sqrt{|\xi|^2+m^2R^{-2}}(\mathcal{F}(\varphi)\big)^2\,d\xi\\
&\le\int_{\mathbb{R}^3}\big(|\xi|+\frac{1}{2}m^2R^{-2}|\xi|^{-1}\big)
\big(\mathcal{F}(\varphi)\big)^2\,d\xi\\
&=\int_{\mathbb{R}^3}\varphi \mathcal{F}^{-1}\big(|\xi|\mathcal{F}(\varphi)\big)\,dx
+\frac{1}{2}m^2R^{-2}\int_{\mathbb{R}^3}\varphi \mathcal{F}^{-1}
\big(|\xi|^{-1}\mathcal{F}(\varphi)\big)\,dx\\
&=\int_{\mathbb{R}^3}\varphi
\Big(\sqrt{-\Delta}+\frac{1}{2}m^2R^{-2}(-\Delta)^{-\frac{1}{2}}\Big)\varphi\, dx.
\end{split}
\]
The assertion follows.
\end{proof}

Next fact enable us to show the positivity of minimizers of problem $e(a)$.
 \begin{Lemma}\label{lem:A2}
 For any $u\in H^{\frac{1}{2}}(\mathbb{R}^3)$, we have
\begin{equation}\label{eq:a2}
 \int_{\mathbb{R}^3}|u|\sqrt{-\Delta+m^2}|u|\,dx\le \int_{\mathbb{R}^3}u\sqrt{-\Delta+m^2}\,u\,dx.
\end{equation}
 \end{Lemma}

\begin{proof}
It is known from Theorem 7.12 in \cite{LL}, see also Theorem 5 in \cite{A} that, for any $u\in H^{\frac{1}{2}}(\mathbb{R}^3)$,
\[
\begin{split}
&\quad\int_{\mathbb{R}^3}u\sqrt{-\Delta+m^2}\,u\,dx\\
&=m\int_{\mathbb{R}^3}|u|^2\,dx
+\big(\frac{m}{2\pi}\big)^2\int_{\mathbb{R}^3}\int_{\mathbb{R}^3}\frac{|u(x)-u(y)|^2}{|x-y|^2}K_2(m|x-y|)\,dxdy,\\
\end{split}
\]
where
\[
K_2(m|x-y|)=\frac{\sqrt{\pi}{\rm e}^{-m|x-y|}}{\sqrt{2m|x-y|}\Gamma(\frac{5}{2})}\int_0^\infty{\rm e}^{-t}\Big(t+\frac{t^2}{2m|x-y|}\Big)^{\frac{3}{2}}\,dt.
\]
This, together with the inequality
$$
||u(x)|-|u(y)||\le|u(x)-u(y)|,
$$
implies \eqref{eq:a2}.
\end{proof}

Although there is no uniqueness results for solutions of \eqref{eq:1.7}, we show that every ground state has the
same $L^2$ norm.
\begin{Lemma}\label{lem:A3}
If $u$ is a nontrivial solution of \eqref{eq:1.7}, then $\|u\|_2\ge\|Q\|_2.$ In particular, every ground state $v$ of \eqref{eq:1.7} satisfies $\|v\|_2=\|Q\|_2$. If $u$ is a nontrivial solution of \eqref{eq:1.7} with $\|u\|_2=\|Q\|_2$, then $u$  is a ground state of \eqref{eq:1.7}.
\end{Lemma}
\begin{proof}
By Lemma 5 in \cite{LenLew},
\begin{align}\label{eq:4.3}
\int_{\mathbb{R}^3}\big|(-\Delta)^{\frac{1}{4}}u\big|^2dx
=\frac{1}{2}\int_{\mathbb{R}^3}(|x|^{-1}*u^2)u^2dx=\int_{\mathbb{R}^3}u^2dx.
\end{align}
Thus,
\[
\begin{split}
I(u)&=\frac{\int_{\mathbb{R}^3}\big|(-\Delta)^{\frac{1}{4}}u\big|^2dx
\int_{\mathbb{R}^3}u^2dx}
{\int_{\mathbb{R}^3}(|x|^{-1}*u^2)u^2dx}\\
&=\frac{1}{2}\int_{\mathbb{R}^3}u^2dx\\
&\ge\inf_{u\in H^{\frac{1}{2}}(\mathbb{R}^3),u\neq0}I(u)\\
&=\frac{1}{2}\int_{\mathbb{R}^3}Q^2dx.
\end{split}
\]
This implies $\|u\|_2\ge\|Q\|_2$.

On the other hand, if $u$ satisfies \eqref{eq:1.7} and $\|u\|_2=\|Q\|_2$, for any nontrivial solution $w$,  we have by \eqref{eq:4.3}
\[
\begin{split}
E(w)&=\frac{1}{2}\int_{\mathbb{R}^3}\big|(-\Delta)^{\frac{1}{4}}w\big|^2dx
+\frac{1}{2}\int_{\mathbb{R}^3}w^2dx-\frac{1}{4}(|x|^{-1}*w^2)w^2dx\\
&=\frac{1}{2}\int_{\mathbb{R}^3}w^2dx\\
&\ge\frac{1}{2}\int_{\mathbb{R}^3}Q^2dx\\
&=\frac{1}{2}\int_{\mathbb{R}^3}u^2dx\\
&=E(u).
\end{split}
\]
This means that $u$ is a ground state.
\end{proof}

\vspace{2mm} \noindent{\bf Acknowledgment}  Jianfu Yang  was supported by
NNSF of China, No:11271170, 11371254;  and GAN PO 555 program of Jiangxi.
Jinge Yang was supported by NNSF of China, No:11426130; and the Project of  Nanchang Institute of Technology, No:2014KJ020.

{\small

\begin{thebibliography}{99}

\bibitem{A} V. Ambrosio, Variational Methods for a Pseudo-Relativistic Schrodinger Equation. PhD Thesis(2015),http://www.fedoa.unina.it/10261/.

\bibitem{Cal}A. P. Calder\'{o}n, Commutators of singular integral operators, {\it Proc. Nat. Acad. Sci. U.S.A.}, 53(1965),  1092--1099.

\bibitem{CoMe}R. R. Coifman and Y. Meyer, On commutators of singular integrals and bilinear singular integrals, {\it Trans. Amer. Math. Soc.}, 212 (1975), 315--331.

\bibitem{DLS} Y. Deng, L. Lu and W.Shuai, Constraint minimizers of mass critical Hartree energy functionals: Existence and mass concentration, {\it API J. Math. Phys.}, 56, 061503 (2015).

\bibitem{ES} A. Elgart and B. Schlein, Mean field dynamics of Boson stars, {\it Comm. Pure Appl. Math.}, 64(4) (2006), 3500--545.

\bibitem{FL} J. Fr\"{o}hlich and E. Lenzmann, Mean-field limit of quantum Bose gases and nonlinear Hartree equation. {\it S¨¦minaire: ¨¦quations aux D\'{e}riv\'{e}es Partielles. 2003-2004}, Exp. No. XIX, 26 pp., S\'{e}min. \'{e}qu. D\'{e}riv. Partielles, \'{e}cole Polytech., Palaiseau, 2004.

\bibitem{FJL}J. Fr\"{o}hlich, B. L. G. Jonsson, and Enno Lenzmann, Boson stars as solitary waves, {\it Comm. Math. Phys.} 274 (2007), no. 1, 1--30.

\bibitem{FL1} R. L. Frank and E. Lenzmann, On ground states for the $L^2$-critical boson star equation, arXiv: 0910.2721.

\bibitem{HL} Q. He and W. Long, The concentration of solutions to a fractional Schr\"{o}dinger eqution,  {\it Z. Angew Math. Phys.}, (2016), 67:9. doi:10.1007/s00033-015-0607-x.

\bibitem{GS}Y. J. Guo and R. Seiringer, On the mass concentration for Bose-Einstein condensates with attactive interactions, {\it Lett. Math. Phys.}, 104(2014), 141--156.

\bibitem{GZ}Y. J. Guo and X. Zeng, Ground state of pseudo-relativistic boson stars under the critical stellar mass, {\it Ann.I.H.Poincare, Anal. Nonli.}, to appear.

 \bibitem{Lenz2} R. L. Frank, E. Lenzmann,  and Silvestre, L. , Uniqueness of Radial Solutions for the Fractional Laplacian. {\it Comm. Pure Appl. Math.}, (2015) doi:10.1002/cpa.21591.

 \bibitem{L} E. H. Lieb, The stability of matter: from atoms to stars, {\it Bull. Amer. Math.}, 22(1990), 1--49.

 \bibitem{LL} E. H. Lieb and M. Loss, Analysis, {\it Graduate Studies in Mathematicas 14,} AMS, 2001.

 \bibitem{LT} E.H. Lieb and  W. Thirring, Gravitational collapse in quantum mechanics with relativistic kinetic energy, {\it Ann. Phys.}, 115(2)(1984), 494--512.

 \bibitem{LY} E.H. Lieb and  H.-T. Yau, The Chandrasekhar theory of stellar collapse as the limit of quantum mechanics, {\it Comm. Math. Phys.}, 112(1987), 147--174.

\bibitem{Lenz3} E. Lenzmann,  Well-posedness for semi-relativistic Hartree equations of critical type, {\it Math.Phys. Anal. Geom.}, 10 (2007),43-64.

\bibitem{LenLew}E. Lenzmann and M. Lewin, On singularity formation for the $L^2$-critical Boson star equation, {\it Nonlinearity}, 24 (2011), 3515--3540.

\bibitem{N}D.T.Nguyen, Blow-up profile of ground states for the critical boson star, {\it Arxiv:}1703.10324v1 [math-ph] 30 Mar 2017.

\bibitem{W} M. I. Weinstein, Nonlinear Schr\"{o}dinger equations and sharp interpolation estimates, {\it Comm. Math. Phys.}, 87 (1983), 567--576.


\end{thebibliography}
\end{document}